\documentclass[10 pt,a4paper,reqno]{amsart}

\usepackage[utf8]{inputenc}

\usepackage{array,amsbsy,amscd,amsfonts,color,amssymb,amstext,amsmath,latexsym,amsthm,amscd,multicol,graphicx, tikz, physics}
\usepackage[english]{babel}
\usepackage{hyperref}
\usepackage[mathscr]{eucal}
\usepackage{tikz-cd}

\newtheorem{theorem}{Theorem}
\numberwithin{theorem}{section}
\newtheorem{proposition}[theorem]{Proposition}
\newtheorem{lemma}[theorem]{Lemma}
\newtheorem{corollary}[theorem]{Corollary}

\theoremstyle{definition}
\theoremstyle{definition}\newtheorem{definition}[theorem]{Definition}
\theoremstyle{definition}\newtheorem{remark}[theorem]{Remark}
\theoremstyle{definition}\newtheorem{example}[theorem]{Example}
\theoremstyle{definition}
\theoremstyle{definition}
\theoremstyle{definition}
\theoremstyle{definition}\newtheorem{problem}[theorem]{Problem}
\theoremstyle{definition}
\theoremstyle{definition}

 \newcommand{\se}{\subseteq}
\newcommand{\T}{\mathbb{T}}
\newcommand{\R}{\mathbb{R}}
\newcommand{\C}{\mathbb{C}}
\newcommand{\A}{\mathcal{A}}
\newcommand{\CC}{\mathcal{C}}
\newcommand{\B}{\mathcal{B}}
\newcommand{\M}{\mathcal{M}}

\newcommand{\K}{\mathcal{K}}
\newcommand{\G}{\mathcal{G}}
\newcommand{\E}{\mathcal{E}}
\newcommand{\F}{\mathcal{F}}

\renewcommand{\H}{\mathcal{H}}

\newcommand{\SB}{\mathscr{B}}
\newcommand{\BH}{\mathscr{B}(\mathcal{H})}
\newcommand{\BK}{\mathscr{B}(\mathcal{K})}
\newcommand{\BG}{\mathscr{B}(\mathcal{G})}
\newcommand{\Ba}{\SB^a}
\newcommand{\BaE}{\mathscr{B}^a(E)}

\newcommand{\la}{\langle}
\newcommand{\ra}{\rangle}
\newcommand{\1}{\mathbf{1}}

\newcommand{\IG}{\I\!\!\Gamma}
\newcommand{\ilim}{\lim\operatorname{ind}}

\newcommand{\I}{\mathbb{I}}
\newcommand{\J}{\mathbb{J}}
\newcommand{\ft}{{\mathfrak t}}
\newcommand{\fs}{{\mathfrak s}}
\newcommand{\fr}{\mathfrak r}

\DeclareMathOperator{\spn}{span}

\DeclareMathOperator{\id}{id}

\title{Structure of Block Quantum Dynamical Semigroups and their Product Systems}
\author{B.V. Rajarama Bhat}
\address{Indian Statistical Institute, Stat-Math. Unit, R V College Post, Bengaluru 560059, India}
\email{bhat@isibang.ac.in}
\author{Vijaya Kumar U.}
\address{Indian Statistical Institute, Stat-Math Unit, R V College Post, Bengaluru 560059, India}
\email{uvijayakumar87@gmail.com }


\begin{document}

    \begin{abstract}
        W. Paschke's version of Stinespring's theorem associates a Hilbert $C^*$-module along with a generating vector
        to every completely positive map. Building on this,  to every quantum dynamical semigroup (QDS) on a
        $C^*$-algebra $\A$ one may associate
         an inclusion system $E=(E_t)$ of Hilbert $\A$-$\A$-modules with a generating
          unit $\xi =(\xi _t)$. Suppose $\B$ is a von Neumann algebra, consider $M_2(\B)$, the von Neumann algebra of $2\times 2$ matrices with
          entries from $\B$. Suppose $(\Phi_t)_{t\ge 0}$ with
           $\Phi_t=\begin{pmatrix}
        \phi_t^1& \psi_t\\\psi_t^*&\phi_t^2
        \end{pmatrix},$ is a QDS on $M_2(\B)$ which acts block-wise and let   $(E^i_t)_{t\ge 0}$ be the inclusion
          system associated to the diagonal QDS $(\phi^i_t)_{t\ge 0}$ with the generating unit $(\xi_t^i)_{t\ge 0}, i=1,2.$ It is shown that
           there is a contractive (bilinear)        morphism $T=(T_t)_{t\ge0}$
         from $(E^2_t)_{t\ge 0}$ to $(E^1_t)_{t\ge 0}$ such
         that $\psi_t(a)=\la \xi^1_t, T_t a\xi^2_t\ra$ for all $a\in\B.$ We
          also prove that any contractive morphism between inclusion systems of von Neumann $\B$-$\B$-modules can be lifted as a
           morphism between the product systems generated by them. We observe that the $E_0$-dilation of a block quantum
           Markov semigroup (QMS) on a unital $C^*$-algebra is again a semigroup of
           block maps.
    \end{abstract}

    \maketitle

    \section{Introduction}
    It is well-known that a block matrix $\begin{pmatrix}
    A&B\\B^*&D
    \end{pmatrix}$ of operators on a direct sum of Hilbert spaces $\H \oplus \K$ is positive
    if and only if $A, D$ are positive and
    there exists a contraction $K: \K \to \H$ such that $B=A^{\frac{1}{2}}KD^\frac{1}{2}.$
    This says that the positivity of a block matrix is determined up to a contraction by the positive diagonals.
     We want to look at the structure of  block completely positive (CP) maps, that is, completely positive maps which send $2\times 2$ block
    operators as above to $2\times 2$ block operators.  Such maps   have already appeared  in many different
    contexts.
     For example, Paulsen uses the block CP maps in \cite{Paul} to prove
    that every completely polynomially bounded operator is similar to a contraction. The structure of completely bounded (CB)
    maps are understood using the $2\times 2$ block CP maps (See \cite{Paul,PS,Suen},\cite[Chapter 8]{Paulsen-book}).  The
    usual way to study the structure of CP maps into $\BH$ is via Stinespring dilation theorem (\cite{Stine}),
    which says that if $\phi:\A\to\BH$ is a CP map then there is a triple $(\K,\pi,V)$ of a Hilbert space $\K$, a
    representation $\pi:\A\to \BK$ and a bounded operator $V\in\SB(\H,\K)$ such that $\phi(a)=V^*\pi(a)V$ for all $a\in \A.$
    If  $\Phi=\begin{pmatrix}
    \phi_1&\psi\\\psi^*&\phi_2
    \end{pmatrix}:M_2(\A)\to M_2(\BH)$ is a block CP map, then the diagonals $\phi_i,i=1,2$ are also CP  maps on $\A.$ Also the Stinespring representation of $\Phi$ gives us natural Stinespring representations for $\phi_i$ by the appropriate compressions. In \cite[Corollary 2.7]{PS} Paulsen and Suen proved that: if   $\Phi=\begin{pmatrix}
    \phi&\psi\\\psi^*&\phi \end{pmatrix}:M_2(\A)\to M_2(\BH)$  is CP and if $\phi$ has the
     minimal Stinespring representation $(\K, \pi,V)$ then there exists a contraction $T\in \pi(\A)' $ such
     that $\psi(\cdot)=V^*\pi(\cdot)TV.$ In \cite{KF} Furuta studied the completion problems of partial matrices of block completely positive maps (see Remark \ref{rem-Furuta}). These results show the
     importance of studying block CP maps. In this article, we want
     to study one parameter semigroups of block CP maps.

    To begin with, the Stinespring's theorem approach is not so convenient
    to study compositions of CP maps between general $C^*$-algebras. In \cite{Paschke} Paschke proved a structure
    theorem for CP maps between $C^*$-algebras, which is a
    generalization of Stinespring theorem,  we call it as GNS-construction for CP
    maps, in view of its close connection with the familiar GNS-construction for
    positive linear functionals on $C^*$-algebras. This theorem
      says that given a CP map
    $\phi:\A\to\B$ there is a pair $(E,\xi)$ of a Hilbert $\A$-$\B$-module $E$ (see below for definitions)
     and a cyclic vector $\xi$ in $E$
    such that $\phi(a)=\la \xi,a\xi\ra$ for all $a\in \A.$ The advantage of GNS-construction is that
    we can write the GNS-construction of the composition of two CP maps as a submodule of the tensor product
    of their GNS-constructions (see Remark \ref{comp}).  Employing this scheme we can associate an inclusion
    system  (synonymous with subproduct system)  $(E_t)$ of Hilbert $\B$-$\B$-modules with a generating unit $(\xi_t)$ to
    any one parameter semigroup of  completely positive maps on $\B.$ It may be recalled
    that semigroups of completely positive maps on a
    $C^*$-algebra are known as quantum dynamical semigroups (QDS), semigroups of unital completely positive maps are known as quantum Markov semigroups (QMS) and
    semigroups of unital endomorphisms are known as
    $E_0$-semigroups. In  \cite{Bh-ind} Bhat proved that any QMS on $\BH$
    admits a unique $E_0$-dilation, and  in \cite{Bh-dil} extended the result
    to QMS on unital  $C^*$-subalgebras of $\BH $.
    Later in \cite{BS} Bhat and Skeide constructed the $E_0$-dilation for arbitrary quantum Markov semigroups (QMS)
    on abstract unital $C^*$-algebras,  using the
    technology of Hilbert $C^*$-modules.  Here one sees for the first time  subproduct systems and product systems
    of Hilbert $C^*$-modules.
     Muhly and Solel (in \cite{MS})
    took a dual approach to achieve this, where they have called
    these Hilbert $C^*$-modules as $C^*$-correspondences.

      While studying units of  $E_0$-semigroups of $\BH$ Powers was led into considering block CP semigroups
      (See  \cite{Po} and \cite{BLS}, \cite{Sk}). In \cite{BM} Bhat and Mukherjee proved a structure theorem for block    QMS on $\SB(\H\oplus \K)$. The main point is that when we have a block QMS,
     there is a contractive morphism between  inclusion systems of
     diagonal CP semigroups. Moreover, this morphism lifts to
     associated product systems. The main goal of this paper is to explore
    the structure of  block quantum dynamical semigroups  on general von Neumann algebras. The
    extension of the theory from $\BH$ case is not straightforward for the following reason. In the case
    of $\BH $, we need only to consider product systems of Hilbert spaces,  whereas now we need to deal with
    both product systems of Hilbert $\B$-modules and also product systems of
    Hilbert $M_2(\B)$-modules (see  Theorem \ref{main-single}) and their inter-dependences. But a careful analysis of
    these modules does lead us to a morphism between inclusion
    systems as in the $\BH$ case and this morphism can also be
    lifted to a morphism at the level of associated product systems
    (Theorem \ref{thm-lifting}). At various steps we consider adjoints of maps
    between our modules and so it is convenient to have von Neumann
    modules. The picture is unclear for Hilbert $C^*$-modules.

    In Section \ref{sec-prelm} we recall the preliminaries. In Section \ref{sec-blk-CP} we prove a
    structure theorem for block CP maps from $M_2(\A)$ to $M_2(\B)$ when $\B$ is a von Neumann algebra also
    we give an example to indicate that we can not replace $\B$ by an arbitrary $C^*$-algebra. We extend this result to semigroups of block CP maps on $M_2(\B)$ in Section \ref{subsec-1}. In Section \ref{subsec-2} we show that the $E_0$-dilation of a block QMS is again block semigroup. In the final section we prove that any morphism   between inclusion systems of von Neumann $\B$-$\B$-modules can be lifted as a morphism between the product systems
     generated by them. Subproduct systems and inclusion systems are synonyms.  The word `subproduct systems' seems to be better established now.   Since we are mostly following the ideas and notations of \cite{BM}, we will continue to call these objects as inclusion systems.

      Given a linear map  $\psi:\A \to \B$ between $C^*$-algebras,  $\psi^*$ denotes  the linear map  $\psi^*:\A\to \B$  defined by $\psi^*(a)=\psi(a^*)^*$ for all $a\in\A.$ For $h\in\H,k\in\K,$ $\ketbra{k}{h}$ denotes the bounded linear operator $\ketbra{k}{h}:\H\to \K$ defined by $\ketbra{k}{h}(h')=\la h,h'\ra k$ for all $h'\in\H.$  All our Hilbert spaces  are taken as complex and separable, with scalar products linear in the second variable.

    \section{Preliminaries}\label{sec-prelm}

    For the basic theory about Hilbert $C^*$-modules we refer to \cite{Lance,BS,BBLS,Paschke,MS-book}. Here we recall only
    the most required definitions and results about Hilbert $C^*$-modules. A {\em pre-Hilbert $\B$-module} $E$ is a vector space, which is a right $\B$-module (compatible with scalar multiplication) together with a $\B$-valued inner product. It is said to be a {\em Hilbert $\B$-module} if it is complete in the norm given by $\norm{x}=\sqrt{\norm{\la x,x\ra}}$ for $ x\in E.$ If we have just semi-inner product
    on a $\B$-module $E,$ using the Cauchy-Schwarz inequality $\la x,y\ra\la y,x\ra\le \norm{\la y, y\ra} \la x,x\ra,$
    we can see that $N=\{x\in E: \la x,x\ra =0\}$ is a $\B$-submodule and hence  $E/N$ is a pre-Hilbert  $\B$-module equipped with the
    natural inner product.  A {\em Hilbert  $\A$-$\B$-module} $E$ is a Hilbert $\B$-module with a nondegenerate
    action of $\A$ on $E.$ 
    Two Hilbert $\A$-$\B$-modules $E$ and $F$ are said to be {\em isomorphic} if there is a bilinear unitary between them and in this case, we write $E\simeq F.$

    \begin{definition}
        Let $E$ be a Hilbert $\A$-$\B$-module and  $F$ be a Hilbert $\B$-$\CC$-module. Then
         $\la x\otimes y,x'\otimes y'\ra=\la y,\la x,x'\ra y'\ra$ defines a semi-inner product on
         (the algebraic tensor product) $E\otimes F$ with the natural right $\CC$-action.
          Let $N=\{w\in E\otimes F: \la w,w\ra=0\}.$  The \emph{interior tensor product} of $E$ and $F$ is
           defined as the completion of ${E\otimes F}/N$ and it is denoted by $E\odot F.$
\end{definition}

 Observe that $E\odot F$ in previous definition
 is a Hilbert $\A$-$\CC$-module with the natural left action of $\A.$ We denote the equivalence class
 of $x\otimes y$ in $E\odot F$ by $x\odot y.$ It may be noted that
 for $b\in \B$, $xb\odot y = x\odot by .$ Let $E,E'$ be Hilbert $\A$-$\B$-modules and $F,F'$ be Hilbert $\B$-$\CC$-modules.
  If $T:E\to E'$ and $S: F\to F'$ are bounded bilinear maps then, $T\odot S:E\odot F\to E'\odot F'$ is a bounded
   bilinear map defined by $(T\odot S)(x\odot y)=Tx\odot Sy$ for $x\in E, y\in F.$

Let $\A$ and $\B$ be $C^*$-algebras and let $\G$ be a Hilbert space on which $\B$ is represented
nondegenerately ($\G$ can be viewed as a Hilbert $\B$-$\C$-module). Let $E$ be a Hilbert $\A$-$\B$-module.
Consider the tensor product   $\H=E\odot \G,$ which is a  Hilbert $\A$-$\C$-module. That is, $\H$ is a Hilbert
space with a representation $\rho:\A\to\BH.$  The map $\rho$ is called the \emph{Stinespring representation} of $\A$ associated
 with $E$ and $\G$ (see Remark \ref{Stinepring rep} below). For $x\in E$ let $L_x:\G\to \H$ be defined by $L_x(g)=x\odot g,$
then $L_x\in \SB(\G,\H)$  with $L_x^*: x'\odot g\mapsto \la x, x'
\ra g.$   Define $\eta: E\to \SB(\G,\H)$ by $\eta(x)=L_x.$ Then  we
have $L_x^*L_y=\la x,y\ra\in \B\subseteq \SB(\G),$  hence,
if the representation of $\B$ on $\G$ is faithful then so is $\eta.$
Also we have $L_{axb}=\rho(a)L_xb$ so that we may identify $E$ as a
concrete subset of $\SB(\G,\H).$ The map $\eta$ is called the \textit{Stinespring representation of $E$}
(associated with $\G$).

In particular, if $\B$ is a von Neumann algebra on a Hilbert space $\G,$ we always consider $E$ as a concrete subset of $\SB(\G, E\odot \G).$

\begin{definition}
 Let $\B$ be a von Neumann algebra on a Hilbert space $\G.$ A Hilbert $\B$-module $E$ is
 said to be a \textit{von Neumann $\B$-module} if $E$ is strongly closed in $\SB(\G, E\odot \G).$ Further, if  $\A$ is a von Neumann algebra, a von Neumann $\B$-module
$E$ is said to be a \textit{von Neumann $\A$-$\B$-module} if it is a
Hilbert $\A$-$\B$-module such that the Stinespring representation
$\rho:\A\to \SB(E\odot \G)$ is normal.
\end{definition}

\begin{remark}
    Let $\A$ be a $C^*$-algebra and $\B$ be a von Neumann algebra on a     Hilbert space $\G$. Let $E$ be a  Hilbert $\A$-$\B$-module. Then $E$ can be completed in strong operator topology to get,      $\overline{E}^s$  which is a Hilbert $\A$-$\B$-module and  is a  von Neumann $\B$-module. Here the left action by $\A$ need not
    be normal.
\end{remark}

\begin{remark}
     If $E$ is a von Neumann $\B$-module, then $\BaE$ is a von Neumann subalgebra
      of $\SB(E\odot \G).$ von Neumann modules are self-dual and hence any bounded right linear map
      between von Neumann modules is adjointable. If $F$ is a von Neumann submodule of $E,$ then
       there exists a projection $p$ $(p=p^2=p^*)$ in $\BaE$ onto
       $F$, that is, $p(E)=F$ and $E$ decomposes as $E= F\oplus F^{\perp
       }.$
\end{remark}

 Let $\A$ and $\B$ be unital $C^*$-algebras and let $\phi:\A\to \B$ be a CP map. Consider the algebraic tensor product $\A\otimes \B.$ For $a,a'\in \A, b,b'\in \B,$ define
 $\la a\otimes b,a'\otimes b'\ra=b^*\phi(a^*a')b.$
 Then $\la\cdot, \cdot\ra$ is a semi-inner product on $\A\otimes\B.$ Let
 $N_{\A\otimes \B}=\{w\in \A\otimes \B: \la w,w\ra=0\}.$ Let $E$ be the completion
 of ${\A\otimes\B}/N_{\A\otimes\B}.$ Then $E$ is a Hilbert $\A$-$\B$-module in
 a natural way. Let $\xi=\1\otimes\1+N_{\A\otimes \B},$ then we have $\phi(a)=\la \xi,a\xi\ra.$  Moreover, $\xi$ is cyclic  (i.e., $E=\overline{\spn}(\A\xi\B)$). The pair $(E,\xi)$ is called the \textit{GNS-construction} of $\phi$ and $E$ is called the \textit{GNS-module} for $\phi.$ Obviously, $\phi$ is unital if and only if $\la \xi,\xi\ra=\1.$

\begin{definition}
    Let $\phi:\A\to \B$  be a CP map. Let $E$ be a Hilbert $\A$-$\B$-module and $\xi\in E,$ We
    call $(E,\xi)$ as a \emph{GNS-representation} for $\phi$ if $\phi(a)=\la \xi, a\xi\ra$ for all $a\in \A.$ It is said to be \emph{minimal} if
    $E=\overline{\spn}(\A\xi\B).$
\end{definition}

Note that the GNS-module in the GNS-construction is minimal. If
$(E,\xi)$ and $(F,\zeta)$ are two minimal GNS-representations for
$\phi$ then the map $\xi\mapsto\zeta$ extends as a bilinear unitary
from $E$ to $F.$ Hence the GNS-representation is unique up to
(unitary) isomorphism.
\begin{remark}\label{Stinepring rep}
    Let $\phi:\A\to \BG$ be a CP map.

    Suppose $(E,\xi)$ is the GNS-construction for $\phi.$ Let $\eta:E\to \SB(\G,\H)$ be the Stinespring representation of $E$ as defined above, then $$\phi(a)=\la\xi,a\xi\ra=L_\xi^*L_{a\xi}=L_\xi^*\rho(a)L_\xi.$$   Note that  $L_\xi$ is an isometry in $\SB(\G,\H)$ if and only if $\phi$ is unital. Note also that $\overline{\spn}\{\rho(a)L_\xi g:a\in\A,g\in\G\}=\overline{\spn}\{a\xi\odot g:a\in\A,g\in\G\}=E\odot \G=\H.$ So we obtain the usual minimal Stinespring representation $(\H,\rho,L_\xi)$ of $\phi.$

      Conversely, if $(\H,\pi, V)$ is the minimal  Stinespring  representation for $\phi.$ Consider $\SB(\G,\H)$ as a Hilbert $\A$-$\BG$-module, where the left action of $\A$ is given by the representation $\pi.$ Let  $E=\overline{\spn}~\A V\BG \se \SB(\G,\H).$ Then $(E,V)$ is a minimal GNS-representation for $\phi.$ 
\end{remark}

\begin{proposition}\label{prop-vNm-1}
    If $E$ is the GNS-module of a normal completely positive map $ \phi:\A\to \B$ between von Neumann algebras,
    then $\overline{E}^s$ is a von Neumann $\A$-$\B$-module.
\end{proposition}

\begin{proposition}\label{prop-vNm-3}
Let $E$ be a von Neumann $\A$-$\B$-module and let $F$ be a von Neumann $\B$-$\CC$-module where $\CC$ acts on a
 Hilbert space $\G.$ Then the strong closure $\overline{E\odot F}^s$ of the tensor product $E\odot F$ in $\SB(\G, E\odot F\odot \G),$ is a von Neumann $\A$-$\CC$-module.
\end{proposition}

\begin{definition}Due to Propositions \ref{prop-vNm-1},  and \ref{prop-vNm-3} we make the following conventions:
    \begin{enumerate}\label{conven}
        \item Whenever $\B$ is a von Neumann algebra and $\phi:\A\to \B$ is a CP map, by GNS-module we always mean $\overline{E}^s,$ where $E$ is the GNS-module, constructed above.
        \item If $E$ and $F$ are von Neumann modules, by tensor product of $E$ and $F$ we mean the strong closure $\overline{E\odot F}^s$ of $E\odot F$ and we still write $E\odot F.$
    \end{enumerate}
\end{definition}

\begin{remark}\label{comp}
		Let $\phi: \A\to\B$ and $\psi:\B\to \CC$ be CP maps with GNS-representations $(E,\xi)$ and $(F,\zeta)$ respectively. Let $(K, \kappa)$ be the GNS-construction of $\psi\circ \phi.$  Note that
	\begin{equation}
	\la \xi\odot\zeta, a\xi\odot\zeta\ra=\la \zeta,\la\xi,a\xi\ra\zeta\ra=\la \zeta,\phi(a)\zeta\ra=\psi\circ\phi(a)\quad \text{for all }a\in \A.
	\end{equation}
	This says that $(E\odot F ,\xi\odot \zeta)$ is a GNS-representation (not necessarily minimal) for $\psi\circ\phi.$
	Thus the the mapping
	\begin{equation}\label{eq-iso}
	\kappa\mapsto \xi\odot \zeta,
	\end{equation}
	extends as a unique bilinear isometry from $K$ to $E\odot F.$ Hence we may identify $K$ as the submodule  $\overline{\spn}(\A\xi\odot\zeta\CC)$ of $E\odot F.$
	
	Note that $E\odot F=\overline{\spn}(\A\xi\B\odot\B\zeta\CC)=\overline{\spn}(\A\xi\odot\B\zeta\CC)=\overline{\spn}(\A\xi\B\odot\zeta\CC).$
	
\end{remark}

In the following we define quantum dynamical semigroups and see
their connection with inclusion systems. We may take $\T $ as either the
semigroup of non-negative integers $\mathbb {Z}_+$ or as the
semigroup of non-negative reals $\mathbb {R}_+$ under addition, but
our real interest lies  in $\T= \mathbb {R} _+$, in view of quantum
theory of open systems.  For more details on this theory look at
\cite{Arv-book, Arv, BS}.
\begin{definition}
    Let $\A$ be a unital $C^*$-algebra. A family $\phi=(\phi_t)_{t\in \T}$ of CP maps on $\A$ is said to be a \textit{quantum dynamical semigroup} (QDS) or \textit{one-parameter CP-semigroup} if
    \begin{enumerate}
        \item $\phi_{s+t}=\phi_s\circ \phi_t$ for all $t\in \T,$
        \item $\phi_0(a)=a$ for all $a\in \A,$
        \item $\phi_t(\1)\le \1$  for all $t\in\T,$ (contractivity)
    \end{enumerate}

    It is said to be  \textit{conservative QDS} or \textit{quantum Markov semigroup} (QMS) if $\phi_t$ is unital for all $t\in\T.$ In practice, in addition to (1)-(3) we may assume continuity of $t\to \phi_t(a)$ in different topologies, depending upon the context.
\end{definition}

    \begin{definition}\label{inc-sys}
        Let $\B$ be a $C^*$-algebra. An \textit{inclusion system} $(E,\beta)$ is a  family $E=(E_t)_{t\in \T}$ of Hilbert $\B$-$\B$-modules with  $E_0=\B $ and  a family $\beta=(\beta_{s,t})_{s,t\in \T}$ of bilinear isometries $\beta_{s,t} :E_{s+t}\to E_s\odot E_t$ such that, for all $r,s,t\in \T,$
        \begin{equation}
        (\beta_{r,s}\odot \id_{E_t})\beta_{r+s,t}=(\id_{E_r}\odot\beta_{s,t})\beta_{r,s+t} .            \end{equation}
        It is said to be  a \textit{product system} if every $\beta_{s,t} $ is unitary.
    \end{definition}

\begin{remark}
    If $\B$ is  von Neumann algebra in Definition \ref{inc-sys}, then we consider inclusion system of von Neumann $\B$-$\B$-modules.
\end{remark}

\begin{definition}
	
	Let $(E,\beta)$ be an inclusion system. A family $\xi^\odot=(\xi_t)_{t\in\T}$ of vectors $\xi_t\in E_t$ is called a \emph{unit} for the inclusion system, if $\beta_{s,t}(\xi_{s+t})=\xi_s\odot\xi_t.$ It is said to be  \emph{unital} if $\la \xi_t,\xi_t\ra=1$ for all $t\in\T,$ and \emph{generating} if $\xi_t$ is cyclic in $E_t$ for all $t\in\T.$
	Suppose $(E,\beta)$ is a product system, a  unit $\xi^\odot=(\xi_t)_{t\in\T}$ is said to be a \emph{generating unit for the product system} $(E,\beta)$ if $E_t$ is spanned by images of elements $b_n\xi_{t_n}\odot\dots \odot b_1\xi_{t_1}b_0$ ($t_i\in\T,\sum t_i=t, b_i\in \B$) under successive applications of appropriate mappings $\id\odot\beta_{s,s'}^*\odot \id.$
\end{definition}

        Suppose $(E,\beta)$ is an inclusion system with  a  unit $\xi^\odot.$  Consider $\phi_t:\B\to\B$ defined
        by $$\phi_t(b)=\la \xi_t, b\xi_t\ra \text{ for } b\in \B.$$ Then as $\beta_{s,t}$'s are bilinear isometries
        and $\xi^\odot$ is a unit, for $b\in\B$ we have  $$\phi_t\circ\phi_s(b)=\phi_t(\la \xi_s,b\xi_s\ra)=\la
        \xi_t,\la \xi_s,b\xi_s\ra\xi_t\ra=\la \xi_s\odot\xi_t,b(\xi_s\odot\xi_t)\ra=\la \xi_{t+s},b\xi_{t+s}\ra=\phi_{t+s}(b).$$ That
         is,  $(\phi_t)_{t\in\T}$ is a QDS. Clearly  $(\phi_t)_{t\in\T}$ is a QMS if  $\xi^\odot$ is unital.
In the converse direction  we have the following remark:

    \begin{remark}\label{eg-incl-sys}
        Let $\phi=(\phi_t)_{t\in\T}$ be a QDS on a unital $C^*$-algebra $\B$ and let $(E_t,\xi_t)$ be the (minimal) GNS-construction for $\phi_t.$ (Recall that $\xi_t$ is a cyclic vector in $E_t$ such that $\phi_t(b)=\la \xi_t,b\xi_t\ra$ for all $b\in\B$).
        Note that $E_0=\B$ and $\xi_0=\1.$ Define $\beta_{s,t}:E_{s+t}\to E_s\odot E_t$ by
        \begin{equation}
        \xi_{t+s}\mapsto \xi_s\odot\xi_t.\label{eqn-inc-sys}
        \end{equation}
        Then by Remark \ref{comp} $\beta_{s,t}$'s are bilinear isometries. Now
        \begin{align*}
        (\beta_{r,s}\odot I_{E_t})\beta_{r+s,t}(\xi_{r+s+t})&=(\beta_{r,s}\odot I_{E_t})(\xi_{r+s}\odot \xi_t)
        =(\xi_r\odot\xi_s)\odot \xi_t\\
        &=\xi_r\odot(\xi_s\odot \xi_t)=(I_{E_r}\odot \beta_{s,t})(\xi_r\odot \xi_{s+t})\\
        &=(I_{E_r}\odot \beta_{s,t})\beta_{r,s+t}(\xi_{r+s+t})
        \end{align*}
        shows that $(E=(E_t),\beta=(\beta_{s,t}))$ is an inclusion system of Hilbert  $\B$-$\B$-module. It is obvious that $\xi^\odot=(\xi_t)$ is a generating unit for $(E,\beta).$

        Suppose  $\B$ is a von Neumann algebra and  each $\phi_t$ is a  \emph{normal} CP map on $\B,$ then recall from  Proposition \ref{prop-vNm-1} (see also Definition \ref{conven}) that the GNS-module $E_t=\overline{E}^s_t$ is a von Neumann $\B$-$\B$-module for all $t\in \T.$ In this case, $(E,\beta)$ is an inclusion system of von Neumann $\B$-$\B$-modules with the generating unit $\xi^\odot.$
    \end{remark}


\begin{definition}
    For a QDS $\phi=(\phi_t)_{t\ge 0}$  on $\B,$ the inclusion system with the generating unit  $(E,\beta,\xi^\odot)$ as given in Remark \ref{eg-incl-sys} is  called the {\em inclusion system associated to $\phi.$} Sometimes we will just  write $(E,\xi^\odot)$ instead of $(E,\beta,\xi^\odot).$ 
\end{definition}

\begin{definition}
    Let $(E, \beta)$ and $(F, \gamma )$ be two inclusion systems. Let $T=(T_t)_{t\in \T}$ be a family of adjointable bilinear maps $T_t:E_t\to F_t,$ satisfying $\norm{T_t}\le e^{tk}$ for some $k\in \R.$ Then $T$ is said to
    be a \textit{morphism} or a \textit{weak morphism} from  $(E, \beta)$ to $(F, \gamma )$
    if every $\gamma _{s,t}$ is adjointable and
    \begin{equation}
    T_{s+t}=\gamma_{s,t}^* (T_s\odot T_t)\beta_{s,t} \text{ for all } s,t\in \T.
    \end{equation}
    It is said to be a \textit{strong morphism} if
    \begin{equation}
    \gamma_{s,t}T_{s+t}= (T_s\odot T_t)\beta_{s,t} \text{ for all } s,t\in \T.
    \end{equation}
\end{definition}

     \section{Block CP maps}\label{sec-blk-CP}

     Let $\A$ be a unital $C^*$-algebra. Let $p\in\A$ be a projection. Set $p' =\1-p.$ Then for every $x\in \A$ we have the following block decomposition:
     \begin{equation}
     x = \begin{pmatrix}
     p x p & px p' \\ p' x p & p' x p'
     \end{pmatrix}\in \begin{pmatrix}
     p\A p&  p\A p' \\ p' \A p & p' \A p'
     \end{pmatrix}.
     \end{equation}

     \begin{definition}
        Let $\A$ and $\B$ be unital $C^*$-algebras. Let $p\in\A$ and $q\in\B$ be projections. We say that a map $\Phi:\A\to\B$ is a \emph{block map} (with respect to $p$ and $q$) if $\Phi$ respects the above block decomposition. i.e., for all $x\in\A$ we have
        \begin{equation}
        \Phi(x) = \begin{pmatrix}
        \Phi(p x p) & \Phi(px p') \\ \Phi(p' x p) & \Phi(p' x p')
        \end{pmatrix}\in \begin{pmatrix}
        q\B q&  q\B q'  \\q'  \B q &q'  \B q'
        \end{pmatrix}.
        \end{equation}
     \end{definition}

     If $\Phi:\A\to\B$ is a block map, then we get the following four maps:
     $\phi_{11}:p \A p\to q\B q,$ $\phi_{12}:p \A p' \to q\B q',\phi_{21}:p'  \A p\to q' \B q,$ and $\phi_{22}:p' \A p'\to q' \B q' .$
     So we write $\Phi$ as
     \[\Phi=\begin{pmatrix}
     \phi_{11} &\phi_{12} \\\phi_{21} &\phi_{22}
     \end{pmatrix}.\]

    \begin{lemma}\label{lem-single}
    	Let $\A$ and $\B$ be unital $C^*$-algebras. For $i=1,2,$  let $\phi_i:\A\to \B$ be a CP map with GNS-representation $(E_i,x_i).$  Suppose $T: E_2\to E_1$ is  an adjointable bilinear contraction and $\psi:\A\to \B$ is given by $\psi(a)=\la x_1, Tax_2\ra.$ Then the block map $\Phi=\begin{pmatrix}
    	\phi_1 & \psi\\\psi^*& \phi_2
    	\end{pmatrix}: M_2(\A)\to M_2(\B)$ is CP.
    \end{lemma}

     \begin{proof}
       	Set $y=Tx_2\in E_1.$ Then
       	\begin{equation*}
       	\Phi \begin{pmatrix}
     	a &b\\c&d
     	\end{pmatrix}
     	= \begin{pmatrix}
     	\la x_1,a x_1 \ra&\la x_1,b y \ra\\\la y, cx_1\ra&\la y, dy\ra
     	\end{pmatrix} +\begin{pmatrix}
     	0&0\\0&\la x_2, d(\id_{E_2}-T^*T)x_2\ra
     	\end{pmatrix}.
     	\end{equation*}
        Clearly $\begin{pmatrix}
     	a &b\\c&d
     	\end{pmatrix}\mapsto\begin{pmatrix}
     	\la x_1,a x_1 \ra &\la x_1,b y \ra\\\la y, cx_1\ra &\la y, dy\ra
     	\end{pmatrix}$ is CP.
     	Since $T$ is an adjointable bilinear contraction, $(\id_{E_2}-T^*T)$ is bilinear and positive.  Hence  $\begin{pmatrix}
     	a &b\\c&d
     	\end{pmatrix}\mapsto\begin{pmatrix}
     	0&0\\0&\la x_2, d(\id_{E_2}-T^*T)x_2\ra
     	\end{pmatrix}$ is CP. Therefore $\Phi$ is CP.
     \end{proof}


        Let $F$ be a Hilbert $M_2(\B)$-module.  Define a right $\B$-module action and a $\B$-valued semi-inner product  $\la \cdot,\cdot\ra_\Sigma$ on $F$ by
        \begin{equation*}\label{eq-right-act}
        x b:=x\begin{pmatrix}
        b  & 0\\0 & b
        \end{pmatrix} \text{ and }\la x,y\ra_\Sigma:=\sum\limits_{i,j=1}^{2}\la x,y\ra_{i,j}\quad\text{for }  x,y\in F, b \in \B.
        \end{equation*}
        where $\la x,y\ra_{i,j}$ denotes the $(i,j)$\textsuperscript{th} entry of $\la x,y\ra\in M_2(\B).$

        Let  $F^{(\B)}$ denote the quotient space $F/N$ where $N=\{x:\la  x,x\ra_\Sigma=0\}.$ (We  denote the coset $x+N$ of $x\in F$ by  $[x]_F$ or  just by $[x]$). Then $F^{(\B)}$ is a pre-Hilbert $\B$-module with right $\B$-action and  inner product given by
        \begin{equation}
        [x]b=[x\begin{pmatrix}
        b  & 0\\0 & b
        \end{pmatrix}] \text{ and }\la [x],[y]\ra=\la x,y\ra_\Sigma=\sum_{i,j=1}^2 \la x,y\ra_{i,j}\quad\text{for }x,y\in F,b\in\B.
        \end{equation}

    \begin{proposition}\label{new module}
        	     If $F$ is a Hilbert (von Neumann) $M_2(\B)$-module, then $F^{(\B)}$ is a Hilbert (von Neumann) $\B$-module. 
    \end{proposition}

      \begin{proof}  Let $F$ be a Hilbert $M_2(\B)$-module. For each  $x\in F,$ we have $[x]=[x\begin{pmatrix}
      1/2&1/2\\1/2&1/2
      	\end{pmatrix}] $  and
      	       \begin{equation}\label{eq-norms-equal}
       \norm{[x]}=\norm{\sum_{i,j=1}^2\la x,x\ra_{i,j}}^\frac{1}{2}
       =\norm{x\begin{pmatrix}
       	1&1\\1&1
       	\end{pmatrix}}.
       \end{equation}
        Consider a Cauchy sequence $([x_n])_{n\ge 1}$ in $F^{(\B)}.$ Set $y_n=x_n\begin{pmatrix}
       1&1\\1&1 \end{pmatrix}\in F.$ Then by \eqref{eq-norms-equal}, $(y_n)_{n\ge 1}$ is a Cauchy sequence in $F.$ Let $y=\lim_{n\to\infty}y_n$ in $F.$ Then $y=y\begin{pmatrix}
       1/2&1/2\\1/2&1/2
       \end{pmatrix}.$ Take $x=\frac{y}{2}.$
       Then, again by using \eqref{eq-norms-equal}, we see that $([x_n])_{n\ge 1}$ converges to $[x]$ in $F^{(\B)}.$ Thus $F^{(\B)}$ is complete.

       Now assume that $F$ is von Neumann $M_2(\B)$-module. Let $\B\se \BG.$ So $F^{(\B)}\se \SB(\G,F^{(\B)}\odot \G)$ and $F\se \SB(\G^2, F\odot \G^2)$ where $\G^2=\G\oplus \G.$  We have for $x\in F, g_1,g_2\in\G,$
       \begin{equation}\label{eq-SOT-norm}
       \norm{[x]\odot (g_1+g_2)}=\left\la g_1+g_2,\sum_{i,j=1}^2\la x,x\ra_{i,j} (g_1+g_2)\right\ra^\frac{1}{2}=\norm{x\begin{pmatrix}
       		1&1\\1&1
       	\end{pmatrix}\odot \begin{pmatrix}
       	g_1\\g_2
       	\end{pmatrix}}.
       \end{equation}
       Using \eqref{eq-SOT-norm}, we can prove as in the above case, that $F^{(\B)}$ is SOT closed in $\SB(\G,F^{(\B)}\odot \G)$ and hence  $F^{(\B)}$ is  a von Neumann $\B$-module.
    \end{proof}
%
%

         Let $F$ be a Hilbert $M_2(\B)$-module. Suppose $F$ has a nondegenerate left action of $\A,$ then  \eqref{eq-norms-equal} implies that the natural left action of $\A$ on  $F^{(\B)}$ given by
        \begin{equation}\label{new-left-act}
        a [x]:=[ax]\quad\text{for } a\in\A, x\in F
        \end{equation}
        is a well defined nondegenerate action.

        \begin{proposition}\label{new-two-sided}
        	        If $F$ is a Hilbert (von Neumann) $\A$-$M_2(\B)$-module, then $F^{(\B)}$ is a Hilbert (von Neumann) $\A$-$\B$-module with the left action defined in \eqref{new-left-act}. 
       \end{proposition}

       \begin{proof}
       If $F$ is a Hilbert $\A$-$M_2(\B)$-module, then clearly $F^{(\B)}$ is a Hilbert $\A$-$\B$-module. We shall prove that if $F$ is a von Neumann $\A$-$M_2(\B)$-module, then  $F^{(\B)}$ is a von Neumann $\A$-$\B$-module. Let $\B\se \BG.$ So $F^{(\B)}\se \SB(\G,F^{(\B)}\odot \G)$ and $F\se \SB(\G^2, F\odot \G^2)$ where $\G^2=\G\oplus \G.$ We must show that the  Stinespring representation  $\rho:\A\to\SB(F^{(\B)}\odot \G)$ of $\A$  given by  $\rho(a)([x]\odot g) =a[x]\odot g$ is normal. 	For any $x\in F,g\in \G,$ a computation similar to \eqref{eq-SOT-norm} implies that
       	\begin{equation}\label{eq-for-normal}
       	\norm{[x]\odot g}=\norm{x\odot \begin{pmatrix}
       		g\\g
       		\end{pmatrix}}.
       	\end{equation}
       	
       	 As the Stinespring representation $\hat{\rho}:\A\to \SB(F\odot \G^2)$ given by $\tilde{\rho}(a)(x\odot \underline{g})= ax\odot \underline{g}$ for $ a\in\A, \underline{g}\in \G^2$  is normal, using    \eqref{eq-for-normal}, we can see that $\rho$ is normal.
       \end{proof}
    \begin{remark}\label{obs-new-module}
    	   Suppose $F$ is a Hilbert (von Neumann) $M_2(\A)$-$M_2(\B)$-module, then we can consider $F$ as a Hilbert (von Neumann) $\A$-$M_2(\B)$-module by considering the left action of $\A$ given by
   \begin{equation}
   ax:=\begin{pmatrix}
   a &0\\0&a
   \end{pmatrix}x\quad \text{for } x\in F, a\in \A.
   \end{equation}
   Therefore, Proposition \ref{new-two-sided} shows that, if $F$ is a Hilbert (von Neumann) $M_2(\A)$-$M_2(\B)$-module, then $F^{(\B)}$ is a Hilbert (von Neumann) $\A$-$\B$-module.
    \end{remark}
      \begin{remark}
 Let $E\se F$ be a $M_2(\B)$-submodule of a $M_2(\B)$-module $F.$ Then
$$E^{(\B)}\simeq \{[x]_F:x\in E\}\se F^{(\B)}.$$
      \end{remark}


 \begin{theorem}\label{main-single}
    Let $\A$ be a unital $C^*$-algebra and $\B$ be a  von Neumann algebra on a Hilbert space $\G.$ For $i=1,2,$ let $\phi_i:\A\to \B$ be a  CP map with a GNS-representation $(F_i, y_i).$  Suppose $\Phi=\begin{pmatrix}
    \phi _1 & \psi\\
    \psi^* & \phi _2
    \end{pmatrix} : M_2(\A)\to M_2(\B)$ is a block CP map for some CB map $\psi:\A\to\B$ then, there is an adjointable bilinear contraction $T:F_2\to F_1 $ such that $\psi(a)=\la y_1, Tay_2\ra $ for all $a\in \A$.
 \end{theorem}

 \begin{proof}
    Let $(E,x)$ be the (minimal) GNS-construction for $\Phi.$ So, $E$ is a von Neumann $M_2(\B)$-module and Hilbert $M_2(\A)$-$M_2(\B)$-module. Let $\mathbb{E}_{ij}:=\1 \otimes E_{ij}$  in $ \mathcal{A}\otimes M_2,$ or $\mathcal{B}\otimes M_2$, depending upon the context,
     where $\{E_{ij}\}$'s are the matrix units in $M_2.$ Set $\hat{E}_i:=\mathbb{E}_{ii} E\se E,i=1,2.$ Then $\hat{E}_i$'s are SOT closed (as $\mathbb{E}_{ii}$'s are projections) $M_2(\mathcal{B})$-submodules of $E$ such that $E=\hat{E}_1\oplus \hat{E}_2.$

    Let  $x_i:=\mathbb{E}_{ii}x\mathbb{E}_{ii}\in \hat{E}_i, i=1,2.$  Clearly $\la x_1,x_2\ra=0.$ Also for $i,j=1,2 $ and $i\ne j,$
    \begin{equation*}
       \norm{x_i-\mathbb{E}_{ii}x}^ 2=\norm{\mathbb{E}_{ii}x\mathbb{E}_{jj}}^2=\norm{\la \mathbb{E}_{ii}x\mathbb{E}_{jj}, \mathbb{E}_{ii}x\mathbb{E}_{jj}\ra}=\norm{\mathbb{E}_{jj}\Phi(\mathbb{E}_{ii})\mathbb{E}_{jj}}=0,
    \end{equation*}
   and \begin{equation*}
       \norm{x_i-x\mathbb{E}_{ii}}^ 2=\norm{\mathbb{E}_{jj}x\mathbb{E}_{ii}}^2=\norm{\la \mathbb{E}_{jj}x\mathbb{E}_{ii}, \mathbb{E}_{jj}x\mathbb{E}_{ii}\ra}=\norm{\mathbb{E}_{ii}\Phi(\mathbb{E}_{jj})\mathbb{E}_{ii}}=0.
       \end{equation*}
    Thus
    \begin{equation}
    x_i=\mathbb{E}_{ii}x=x\mathbb{E}_{ii}, i=1,2  \text{ and  hence }x=(\mathbb{E}_{11}+\mathbb{E}_{22})x=x_1+x_2.\label{eq-main-1}
    \end{equation}
    As $\Phi$ is a block map, for $A\in M_2(\A),$ using \eqref{eq-main-1} we have
    \begin{equation*}
    \Phi(A)=\la x,Ax\ra=\sum_{i,j=1}^2\la x_i,Ax_j\ra=\begin{pmatrix}
    \la x_1,Ax_1\ra_{11}&\la x_1,Ax_2 \ra_{12}\\\la x_2,Ax_1\ra_{21}&\la x_2,Ax_2\ra_{22}
    \end{pmatrix},
    \end{equation*}
    where $\la a,b\ra_{ij}$ denotes the $(i,j)$\textsuperscript{th} entry of $\la a,b\ra\in M_2(\B).$

   Consider the Hilbert $\A$-$\B$-module and von Neumann $\B$-module $E^{(\B)}$ (as described in Remark \ref{obs-new-module}), and consider  the von Neumann $\B$-modules $ \hat{E}_i^{(\B)},i=1,2. $   Observe that $\hat{E}_i$ has a non-degenerate left action of $\A$ given by
   \begin{equation}
   a x:=\begin{pmatrix}
   a &0\\0&a
   \end{pmatrix}x\quad\text{for } a\in\A, x\in \hat{E}_i.
   \end{equation}
    Therefore, Proposition \ref{new-two-sided} shows that  $\hat{E}_i^{(\B)}$ is also a Hilbert $\A$-$\B$-module for $i=1,2.$

    We have  $E^{(\B)}\simeq \hat{E}_1^{(\B)}\oplus \hat{E}_2^{(\B)}$ (via $[y]_E\mapsto[\mathbb{E}_{11}y]_{\hat{E}_1}+ [\mathbb{E}_{22}y]_{\hat{E}_2}$ for $ y\in E$). 
   For $a\in \mathcal{B}$ and $i=1,2$ see that,
    $$\la [x_i],a[x_i]\ra=\sum_{r,s=1}^2\left\la \mathbb{E}_{ii}x, \begin{pmatrix} a&0\\0&a\end{pmatrix}\mathbb{E}_{ii}x\right\ra_{r,s}=\sum_{r,s=1}^2\Phi\left(\mathbb{E}_{ii}\begin{pmatrix} a&0\\0&a\end{pmatrix}\mathbb{E}_{ii}\right)_{r,s}=\phi_i(a).$$
    This shows that $(\hat{E}_i^{(\B)},[x_i])$ is a GNS-representation (not necessarily minimal) for $\phi_i,i=1,2.$
    Define $U:\hat{E}_2^{(\B)}\to \hat{E}_1^{(\B)}$  by $U[w]=[\mathbb{E}_{12}w]$ for all $w\in\hat{E}_2.$  Then, for all $z,w\in \hat{E}_2,$
    \begin{equation*}
    \la U[z],U[w]\ra=\sum_{i,j=1}^2\la \mathbb{E}_{12}z,\mathbb{E}_{12}w\ra_{i,j}=
    \sum_{i,j=1}^2\la z,\mathbb{E}_{21}\mathbb{E}_{12}w\ra_{i,j}=\sum_{i,j=1}^2\la z, w\ra_{i,j}=\la [z],[w]\ra,
    \end{equation*}
    also  for $y\in \hat{E}_1$ we have $\mathbb{E}_{21}y\in \hat{E}_2$ such that
    \begin{equation*}
    U[\mathbb{E}_{21}y]=[\mathbb{E}_{12}\mathbb{E}_{21}y]=[\mathbb{E}_{11}y]=[y].
    \end{equation*}
    Therefore $U$ is a unitary from the von Neumann $\B$-module $\hat{E}_2^{(\B)}$ to the  von Neumann $\B$-module $\hat{E}_1^{(\B)}.$ Now for $a\in\A,w\in \hat{E}_2, $
    \begin{equation*}
    Ua[w]=U[\begin{pmatrix} a &0\\ 0&a\end{pmatrix}w]=[\mathbb{E}_{12}\begin{pmatrix} a &0\\ 0&a\end{pmatrix}w]=[\begin{pmatrix} a &0\\ 0&a\end{pmatrix}\mathbb{E}_{12}w]=a[\mathbb{E}_{12}w]=aU[w].
    \end{equation*}
    Thus $U:\hat{E}_2^{(\B)}\to \hat{E}_1^{(\B)}$ is a bilinear (adjointable) unitary between the Hilbert $\A$-$\B$ modules.

     Let $\tilde{F_i}:=\overline{\spn}^s \A y_i\B\se F_i$ and       $\tilde{E_i}=\overline{\spn}^s\A [x_i] \B\se \hat{E}_1^{(\B)},$ so that $(\tilde{F_i},y_i)$ and $(\tilde{E_i},[x_i])$ are minimal GNS-representations for $\phi_i,i=1,2.$  Therefore, $\tilde{V}_i:\tilde{F}_i\to \tilde{E}_i$ given by $$\tilde{V_i}(ay_ib)=a[x_i]b,\quad a\in\A,b\in\B,$$  extends to a bilinear (adjointable) unitary.  Let
     $V_i:F_i\to\hat{E}_i^{(\B)}$ be the extension of $\tilde{V_i},$ by defining it to be zero on the complement $\tilde{F_i}^\perp$ of $\tilde{F_i}.$ Note that $V_i$  is a bilinear partial isometry with initial space $\tilde{F_i}$ and final space $\tilde{E_i}$ for $i=1,2.$  Take $T:=V_1^*UV_2.$

Now consider, for $a\in\A,$
    \begin{align*}
    \la y_1, Tay_2\ra &=\la y_1, V_1^*UV_2ay_2 \ra=\la V_1 y_1,UV_2ay_2\ra=\la \tilde{V}_1y_1, U\tilde{V}_2 ay_2\ra\\
    &=\la [x_1], Ua[x_2]\ra 
    =\sum_{i,j=1}^2\left\la \mathbb{E}_{11}x,\mathbb{E}_{12}\begin{pmatrix}
    a &0\\0 &a
    \end{pmatrix}\mathbb{E}_{22}x\right\ra _{i,j}\\
    &=\sum_{i,j=1}^2\left\la x, \begin{pmatrix}
    0& a\\0&0
    \end{pmatrix}x\right\ra_{i,j}=\sum_{i,j=1}^2\Phi\begin{pmatrix}
    0&a\\0& 0
    \end{pmatrix}_{i,j}=\sum_{i,j=1}^2\begin{pmatrix}
    0&\psi(a)\\0& 0
    \end{pmatrix}_{i,j}=\psi(a).
    \end{align*}
    This completes the proof.
 \end{proof}

     \begin{remark}[Uniqueness]\label{uniqueness} With the hypothesis and notations of  Theorem \ref{main-single}
        let $T,T' :F_2\to F_1$ be any two adjointable bilinear contractions such that $\psi (a)=\la y_1 ,Tay_2\ra=\la y_1 ,T'ay_2\ra$ for all $ a\in \A, $ then
        \begin{align*}
        \la a_1y_1 b_1 , T(a_2y_2 b_2) \ra
        &=b_1^*\la y_1, T((a_1^*a_2)y_2)\ra b_2\\
        &=b_1^*\la y_1, T'((a_1^*a_2)y_2)\ra b_2\\
        &=\la a_1y_1 b_1 , T'(a_2y_2 b_2) \ra
        \end{align*}  for $a_1,a_2\in \A,b_1,b_2 \in \B$ and hence $P_{\tilde{F}_1}TP_{\tilde{F}_2}=P_{\tilde{F}_1}T'P_{\tilde{F}_2}$ where $P_{\tilde{F}_i}:F_i\to F_i$ is the projection onto $\tilde{F}_i.$ This in particular shows that the contraction $T$ in Theorem \ref{main-single} is \emph{unique} if $F_i$'s are minimal GNS-modules.
     \end{remark}

       \begin{corollary}\label{single-cor}
        Let $\A$ be a unital $C^*$-algebra. For $i=1,2,$ let $\varphi _i:\A\to \BH$ be a completely positive map with the minimal Stinespring representation $(\K_i,\pi_i, V_i).$  Suppose $\Phi: M_2(\A)\to M_2(\BH),$ defined by $\Phi=\begin{pmatrix}
        \phi _1 & \psi\\
        \psi^* & \phi _2
        \end{pmatrix}
        $ is block CP for some CB map $\psi:\A\to \BH,$ then there is a unique contraction $T:\K_2\to \K_1 $ with $\pi_1(a)T=T\pi_2(a)$ for all $a\in \A$ such that $\psi(a)=V_1^*T\pi_2(a)V_2 $ for all $a\in \A$.
     \end{corollary}

      \begin{proof}
      	Given that $(\K_i,\pi_i, V_i)$ is a minimal Stinespring representation for $\phi_i,i=1,2.$ Let $(E_i,V_i)$ be the minimal GNS-representation for $\phi_i, i=1,2$ as explained in Remark \ref{Stinepring rep}, where $E_i=\overline{\spn}^s\pi_i(\A) V_i\BH \se \SB(\H,\K_i).$ By Theorem \ref{main-single}, there exists an adjointable bilinear contraction $\hat{T}:E_2\to E_1$ such that
        \begin{equation}\label{psi}
        \psi(a)=\la V_1, \hat{T}\pi_2(a)V_2\ra=V_1^*\hat{T}\pi_2(a)V_2 \quad\text{for all }a\in\A.
        \end{equation}

         As $(\K_i,\pi_i, V_i)$ is the minimal Stinespring representation for $\phi,$  we have  $\K_i=\overline{\pi_i(\A)V_i\H}.$ Define $T:\K_2\to \K_1 $ by
         \begin{equation}
         T(\pi_2(a)V_2h)=(\hat{T}(\pi_2(a)V_2))h\quad\text{for all }a\in\A,h\in\H.
         \end{equation}
          Let $h$ be a non-zero vector in $\H.$ As $\hat{T}$  is  right $\BH$-linear and  contraction,  we have for  $a\in\A, h\in \H,$
         \begin{align*}
         \norm{\ketbra{(\hat{T}(\pi_2(a)V_2))h}{h}}=\norm{\hat{T}(\pi_2(a)V_2)\ketbra{h}{h}}&=\norm{\hat{T}(\pi_2(a)V_2\ketbra{h}{h})}\\&\le \norm{\pi_2(a)V_2\ketbra{h}{h}}=\norm{\ketbra{(\pi_2(a)V_2)h}{h}}.
         \end{align*}
         This implies
         \begin{equation}
         \norm{(\hat{T}(\pi_2(a)V_2))h} \le \norm{\pi_2(a)V_2h}\quad \text{for all } a\in\A,  h\in \H.
         \end{equation}
         Therefore  $T$ is a well-defined  contraction.  Now as $\hat{T}$ is left $\A$-linear, for all $a,b\in\A$ and $h\in\H,$ we have
         \begin{align*}
         T\pi_2(a)(\pi_2(b)V_2h)=T(\pi_2(ab)V_2h)&=\hat{T}(\pi_2(ab)V_2)h=\hat{T}(\pi_2(a)\pi_2(b)V_2)h\\&=\pi_1(a)\hat{T}(\pi_2(b)V_2)h=\pi_1(a)T(\pi_2(b)V_2h).
         \end{align*}
          Thus $T\pi_2(a)=\pi_1(a)T,$ for all $a\in \A.$  Now \eqref{psi} shows that   $\psi(a)h=V_1^*T\pi_2(a)V_2h$ for all $h\in\H.$ For the uniqueness of $T,$ let $T'$ be another contraction such that $T'\pi_2(a)=\pi_1(a)T'$ and $\psi(a)=V_1^*T'\pi_2(a)V_2$ for all $a\in\A.$ Consider for $a,b\in\A$ and $ h,g\in \H,$
          \begin{align*}
          \la T\pi_2(b)V_2g,\pi_1(a)V_1h\ra&= \la V_1^*T\pi_2(a^*b)V_2g,h\ra\\
          &=\la\psi(a^*b)g,h\ra=\la V_1^*T'\pi_2(a^*b)V_2g,h\ra\\
          &= \la T'\pi_2(b)V_2g,\pi_1(a)V_1h\ra.
          \end{align*}
          This proves the uniqueness of $T.$
       \end{proof}

      \begin{remark}\label{rem-Furuta}
      	(i). Corollary \ref{single-cor} can be proved directly (without deducing from Theorem \ref{main-single}).
      	
        \noindent(ii). Given two CP maps $\phi_i:\A \to\BH, i=1,2. $ Let $(\K_i,\pi_i, V_i)$ be  the  minimal Stinespring representation for $\phi_i , i=1,2.$   Suppose the block map $\Phi=\begin{pmatrix}
     \phi _1 & \psi\\
     \psi^* & \phi _2
     \end{pmatrix}$ is CP for some CB map $\psi:\A\to\BH.$ Then  Furuta in \cite[Proposition 6.1]{KF} proved that: $\psi$ is non-trivial (non-zero) if and only if there exists a non-zero operator $T:\K_2\to\K_1$ such that $T\pi_1(a)=\pi_2(a)T$ for all $a\in \A.$ On the other hand, Corollary \ref{single-cor} explicitly tells us the structure of $\psi$ from the minimal Stinespring representations of  $\phi_i$'s.

     \noindent(iii). Corollary \ref{single-cor} is a generalization of \cite[Corollary 2.7]{PS} (namely, when  $\phi_1=\phi_2$ in  Corollary \ref{single-cor},  we get the result of Paulsen and Suen \cite[Corollary 2.7]{PS}).
    \end{remark}
The following example shows that we cannot replace the von Neumann algebra $\B$ in Theorem \ref{main-single} by an arbitrary $C^*$-algebra.

 \begin{example}\label{eg-corner}
    Let $\A= \B=C([0,1]),$ the commutative unital $C^*$-algebra of continuous functions on $[0,1].$  Let $E=C([0,1]).$ It is a Hilbert $\A$-$\B$-module with the natural actions  and standard inner product: $\langle f, g\rangle = f^*g.$    Let
    \begin{equation*}
    h_1(t)=t,\quad h_2(t)=1 \quad\text{ for }  t\in [0,1].
    \end{equation*}
    Consider the CP map  $\Phi:M_2(\A)\to M_2(\B)$ defined by
    \[\Phi\begin{pmatrix}
    f_{11}&f_{12}\\f_{21}&f_{22}
    \end{pmatrix}=\begin{pmatrix}
    h_1^*&0\\0&h_2^*
    \end{pmatrix}\begin{pmatrix}
    f_{11}&f_{12}\\f_{21}&f_{22}
    \end{pmatrix}\begin{pmatrix}
    h_1&0\\0&h_2
    \end{pmatrix}=\begin{pmatrix}
    h_1^*f_{11}h_1 & h_1^*f_{12}h_2\\  h_2^*f_{21}h_1 &h_2^*f_{22}h_2
    \end{pmatrix}.\]
    Note that $\Phi$ is the block CP map $\begin{pmatrix}
    \phi_1& \psi\\ \psi^*&\phi_2
    \end{pmatrix},$ where $\phi_i,\psi:\A\to \B$ are given by
    \begin{equation}
    \psi(f)=\la h_1,fh_2\ra \text{ and }\phi_i (f)=\la h_i,fh_i\ra\quad\text{for  } f\in\A, i=1,2.
    \end{equation}

    Therefore, $(E,h_i)$ is a GNS-representation for $\phi_i,i=1,2.$ Let  $E_i=\overline{\spn}~\A h_i\B\se E.$ Then $(E_i,h_i)$ is the minimal GNS-representation for $\phi_i,i=1,2.$ Note that
    $$E_1=\{f\in C([0,1]): f(0)=0\} \quad \text{ and } \quad E_2=\A.$$

    Now suppose that there exists a bilinear contraction $T:E_2\to E_1$ such that $\psi(f)=\la h_1, Tfh_2\ra $ for all $f\in \A.$ Then $\psi(h_2)=h_1=h_1T(h_2).$ That is, $t=tT(h_2)(t)$ for all $t\in[0,1].$ This implies that  $T(h_2)(t)=1$ for all $t\ne 0.$ This is a
     contradiction to $T(h_2)\in E_1.$
\end{example}

\begin{remark}
    Let $\A$ and $\B$ be unital $C^*$-algebras and let $\Phi:M_2(\A)\to M_2(\B)$ be a block CP map $\Phi=\begin{pmatrix}
    \phi _1 & \psi\\
    \psi^* & \phi _2
    \end{pmatrix}.$ Suppose $\B$ is a unital subalgebra of $\BH$ for some Hilbert space $\H.$ Let $\CC$ be the von Neumann algebra $\overline{\B}^s.$ Now enlarge the codomain of $\Phi$ to $M_2(\CC).$ That is, consider the block CP map  $\tilde{\Phi}:M_2(\A)\to M_2(\CC),$ such that $\tilde{\Phi}(A)=\Phi(A).$

    Let $\tilde{\Phi}=\begin{pmatrix}
    \tilde{\phi} _1 & \tilde{\psi}\\
    \tilde{\psi}^* & \tilde{\phi} _2
    \end{pmatrix}.$ Then,  by Theorem \ref{main-single} we get a bilinear contraction
     $\tilde{T}:\tilde{E}_2 \to \tilde{E}_1$ such that $\psi(a)=\tilde{\psi}(a)=\la x_1, \tilde{T}ax_2\ra$ for
     all $a\in\A,$ where $(\tilde{E}_i,x_i)$ is the GNS-construction for
     $\tilde{\phi}_i,i=1,2.$ Note that $(\tilde{E}_i,x_i)$ is not a GNS-representation
      for $\phi_i, i=1,2$ as $\tilde{E}_i$ is an Hilbert
      $\A$-$\CC$-module which need not be an Hilbert $\A$-$\B$
      module.

    In particular, in  Example \ref{eg-corner} if we enlarge the codomain $\B$ to
     $\CC=\overline{\B}^s=L^\infty([0,1]),$ then with the above notations, we have
     $\tilde{E}_i=\overline{\spn}~\A h_i\CC=L^\infty([0,1]),i=1,2.$ Note also that there exists
     a bilinear contraction $\tilde{T}:\tilde{E}_2 \to \tilde{E}_1$ given by $\tilde{T}f=f, f\in E_2$ such
     that $\psi(f)=\tilde{\psi}(f)=\la h_1,Tfh_2\ra$ for all $f\in \A.$
\end{remark}

The following example is a modification of Example \ref{eg-corner} to get
 an example of a unital block CP map $\Phi:M_2(\B)\to \M_2(\B).$

\begin{example}
    Let $\A$ be the unital $C^*$-algebra $C([0,1]).$ Let $\B=\A\oplus \A$ and let $ F=\A\oplus \A$ be the
     Hilbert $\B$-$\A$-module with the module actions and inner product  given by
    $$\begin{pmatrix}
    f_1\\f_2
    \end{pmatrix}k=\begin{pmatrix}
    f_1k\\f_2k
    \end{pmatrix}, \begin{pmatrix}
    k_1\\k_2
    \end{pmatrix}\begin{pmatrix}
    f_1\\f_2
    \end{pmatrix}=\begin{pmatrix}
    k_1f_1\\k_2f_2
    \end{pmatrix}\text{ and } \left\la \begin{pmatrix}
    f_1\\f_2
    \end{pmatrix},\begin{pmatrix}
    g_1\\g_2
    \end{pmatrix}\right\ra=f_1^*g_1+f_2^*g_2$$
    for $k\in \A,\begin{pmatrix}
    k_1\\k_2
    \end{pmatrix}\in \B,  \begin{pmatrix}
        f_1\\f_2
    \end{pmatrix},\begin{pmatrix}
    g_1\\g_2
    \end{pmatrix}\in F.$
Consider $E=F\oplus F$ as a Hilbert $\mathcal{B}$-$\B$-module with right action
$$ \begin{pmatrix}  x\\y\end{pmatrix} f=\begin{pmatrix}xf_1\\y  f_2\end{pmatrix} ~\text{ where } f=\begin{pmatrix}  f_1\\f_2\end{pmatrix}\in \B,\begin{pmatrix} x\\y\end{pmatrix}\in E,$$
inner product
$$ \left\la \begin{pmatrix}x_1\\x_2 \end{pmatrix},\begin{pmatrix}y_1\\y_2\end{pmatrix}\right\ra=\begin{pmatrix} \la x_1, y_1\ra\\\la x_2, y_2\ra    \end{pmatrix},$$
and the left action
$$f \begin{pmatrix} x\\y\end{pmatrix} = \begin{pmatrix} f x\\f y\end{pmatrix} \quad\text{for } f\in \B, \begin{pmatrix}
x\\y
\end{pmatrix}\in E.$$

Let $h_{11}(t)=t,h_{12}(t)=\sqrt{1-t^2}, h_{21}(t)=1, h_{22}(t)=0$ for $t\in[0,1].$ Let
$$h_1=\begin{pmatrix}   h_{11}  \\ h_{12}   \end{pmatrix}\oplus \begin{pmatrix}h_{11}   \\ h_{12}   \end{pmatrix}, h_2=\begin{pmatrix}  h_{21}\\h_{22}  \end{pmatrix} \oplus\begin{pmatrix} h_{21}\\h_{22}  \end{pmatrix}\in E=F\oplus F $$

Let $\Phi: M_2(\B)\rightarrow M_2(\B)$ be the block CP map $\Phi=\begin{pmatrix}
\phi_1&\psi\\\psi^*&\phi_2
\end{pmatrix}, $
where $\phi_i , \psi: \B\to \B$ are defined by
$$\phi_i (f)=\la h_i,f h_i\ra\text{ and } \psi(f)=\la h_1,fh_2\ra\quad \text{for } f\in\B,i=1,2. $$

Let $E_i=\overline{\spn}\B h_i\B\se E,i=1,2.$ Then $E_1=F_1\oplus
F_1$ with $F_1=C_0([0,1])\oplus C_1([0,1])$
 where
$C_j([0,1])=\{f\in C([0,1]): f(j)=0\}$ for $j=0,1,$ and
$E_2=F_2\oplus F_2$ with $F_2=C([0,1])\oplus 0.$ Now suppose there
exists a bilinear contraction $T:E_2\rightarrow E_1$ such that
$\psi(f)=\la h_1,Tfh_2\ra$ for all $f\in \B.$ Then for
$f=\begin{pmatrix}
    h_{21}\\h_{22}
    \end{pmatrix}\in \B,$  $$\begin{pmatrix}
    h_{11}\\h_{22}
    \end{pmatrix}=\la h_1,fh_2\ra=\psi(f)=\la h_1,fTh_2\ra=\left\la\begin{pmatrix}  h_{11}  \\ h_{12}   \end{pmatrix}\oplus \begin{pmatrix}h_{11}   \\ h_{12}   \end{pmatrix}, \begin{pmatrix}
    l_{11}\\h_{22}
    \end{pmatrix}\oplus \begin{pmatrix}
    l_{21}\\h_{22}
    \end{pmatrix}\right\ra$$
    where $Th_2=\begin{pmatrix}
    l_{11}\\l_{12}
    \end{pmatrix}\oplus \begin{pmatrix}
    l_{21}\\l_{22}
    \end{pmatrix}\in E_1.$ Therefore $h_{11}=h_{11}l_{11}+h_{12}h_{22}.$ Hence $t=tl_{11}(t)$ for all $t\in [0,1].$ Hence $l_{11}(t)=1$ for $t\ne0.$ This is a contradiction
    to the assumption that $Th_2\in E_1.$ So no such $T$ exists.
\end{example}

We could not get any reasonable answer to the following question.

 \begin{problem}
    Let $\A,\B$ be unital $C^*$-algebras and let $p\in\A,q\in \B$ be projections. Let $\Phi=\begin{pmatrix}
    \phi_1&\psi\\\psi^*&\phi_2
    \end{pmatrix}$ be a block CP map  from $\A$ to $\B$
    with respect to $p$ and $q.$ Let $(E_i,\xi_i)$ be GNS-representation  of  $\phi_i,i=1,2.$  Can we prove a
    theorem similar to Theorem \ref{main-single}? In other words what is the structure of $\psi$ in terms of $(E_i,\xi_i)?$
\end{problem}

     \section{Semigroups of block CP maps}\label{sec-CP-semi}

    \subsection{Structure of block quantum dynamical semigroups}\label{subsec-1}
   In this section, we shall prove a structure theorem similar to (or using) Theorem \ref{main-single} for semigroups
    of block CP maps. We shall start with a few basic examples of semigroups of block CP
    maps, which are of interest.

    \begin{example}
        Let $\H$ be a Hilbert space. Let $(\theta_t)_{t\ge 0}$
        be an $E_0$-semigroup on $\BH.$ Let $(U_t)_{t\ge 0}$ be
        a family of unitaries in $\BH$ forming left cocycle for $\theta $, that is,
        $U_0=I, U_{s+t}= U_s\theta _s(U_t)$, $t\mapsto U_t$ continuous in SOT.
         Let $\psi_t(X)=U_t\theta_t(X)U_t^*$ for $ X\in \BH.$ Then $(\psi_t)_{t\ge 0}$ is an $E_0$-semigroup, cocycle conjugate to $(\theta_t)_{t\ge 0}.$ Define $\tau_t:\SB(\H\oplus \H)\to \SB(\H\oplus \H)$ by $$\tau_t\begin{pmatrix}
        X & Y\\ Z& W
        \end{pmatrix}=\begin{pmatrix}
        I &0\\ 0 & U_t
        \end{pmatrix} \begin{pmatrix}
        \theta_t(X) & \theta_t(Y)\\ \theta_t(Z) &\theta_t(W)
        \end{pmatrix}\begin{pmatrix}
        I &0\\ 0 & U_t^*
        \end{pmatrix}=\begin{pmatrix}
        \theta_t(X) &\theta_t(Y)U_t^* \\ U_t\theta_t(Z) & U_t\theta_t(W)U_t^*
        \end{pmatrix}.$$
        Then clearly $(\tau _t)_{t\ge 0}$ is a block $E_0$-semigroup.
    \end{example}

    \begin{example}
        Let $(a_t)_{t\ge 0}$ and $(b_t)_{t\ge 0}$ be semigroups on a $C^*$-algebra $\B$ and let $(\phi_t^i)_{t\ge 0}, i=1,2,$ be two QDS on $\B$ such that $\phi_t^1(\cdot)-a_t(\cdot)a_t^*$ and $\phi_t^2(\cdot)-b_t(\cdot)b_t^*$ are CP maps for  all $t\ge0.$ Define $\tau_t:M_2(\B)\to M_2(\B)$ by
        $$\tau_t\begin{pmatrix}
        a &b \\c&d
        \end{pmatrix}=\begin{pmatrix}
        \phi_t^1(a)&a_tbb_t^*\\b_tca_t^*&\phi_t^2(d)
        \end{pmatrix}.$$
        Then $\tau_t$ is CP, for all $t\ge 0,$  as
        $$\tau_t\begin{pmatrix}
        a &b \\c&d
        \end{pmatrix}=\begin{pmatrix}
        a_t& 0\\0& b_t
        \end{pmatrix}\begin{pmatrix}
        a& b\\c& d
        \end{pmatrix}\begin{pmatrix}
        a_t^*& 0\\0& b_t^*
        \end{pmatrix}+\begin{pmatrix}
        \phi_t^1(a)-a_taa_t^*& 0\\0& \phi_t^2(d)-b_tdb_t^*
        \end{pmatrix}.$$
        Clearly $(\tau_t)_{t\ge 0}$ is a block QDS.
    \end{example}

     \begin{lemma}\label{lem-easy-side}
        Let $\B$ be a unital  $C^*$-algebra. Given two inclusion systems $(E^i,\beta^i,\xi^{\odot i})$ associated to a pair of CP
        semigroups $\phi^i=(\phi_t^i)_{t\ge 0}, i=1,2$ on $\B$ and a contractive  morphism $T =(T_t): E^2 \to E^1,$ there is a block CP semigroup  $\Phi=(\Phi_t)_{t\ge 0}$ on $M_2(\B)$ such that $\Phi_t= \begin{pmatrix}
        \phi_t^1 & \psi_t\\\psi_t^* & \phi_t^2
        \end{pmatrix}$ and $ \psi_t(a)=\la \xi_t^1,T_t(a\xi_t^2)\ra.$
     \end{lemma}

     \begin{proof}
        Define  $\Phi_t:M_2(\B)\to M_2(\B)$  as the block maps
        $\Phi_t= \begin{pmatrix}
        \phi_t^1 & \psi_t\\\psi_t^* & \phi_t^2
        \end{pmatrix},$ where $ \psi_t(b)=\la \xi_t^1,T_t(b\xi_t^2)\ra.$ Then, as $T_t:E_t^2\to E_t^1$ is an adjointable  bilinear contraction, $\Phi_t$ is CP for all $t\ge 0$ (see the proof of Lemma \ref{lem-single}).
%
        Now we shall show that $\Phi=(\Phi_t)_{t\ge 0}$ is a semigroup on $M_2(\B).$ We have
        $$
        \Phi_s\circ \Phi_t\begin{pmatrix}
        a &b\\c&d
        \end{pmatrix}=\Phi_s\begin{pmatrix}
        \phi_t^1 (a)& \psi_t(b)\\\psi_t^* (c)& \phi_t^2(d)
        \end{pmatrix}
        =\begin{pmatrix}
        \phi_{s+t}^1 (a)& \psi_s(\psi_t(b))\\\psi_s^*(\psi_t^* (c))& \phi_{s+t}^2(d)
        \end{pmatrix}.
        $$
        It is clear from this, that to show  $\Phi$ is a semigroup, it is enough to show that $(\psi_t)_{t\ge 0}$ is a semigroup. Now as $T$ is a morphism it is easy to see that $(\psi_t)_{t\ge 0}$ is a semigroup.
     \end{proof}

        In the following we have the converse of Lemma \ref{lem-easy-side},  when $\B$ is a von Neumann algebra. Example \ref{eg-corner} says that we cannot take $\B$ as an arbitrary $C^*$-algebra.

     \begin{theorem}\label{main-semigroup}
        Let $\B$ be a von Neumann algebra. Let $\Phi=(\Phi_t)_{t\ge 0}$ be a semigroup of block normal CP maps on $M_2(\B)$ with  $\Phi_t=\begin{pmatrix}  \phi_t^1 & \psi_t\\\psi_t^* & \phi_t^2
        \end{pmatrix}.$  Then, there are inclusion systems $(E^i, \beta^i, \xi^{\odot i}), i=1,2$ associated to $\phi^i$ (canonically arising  from the inclusion system associated to $\Phi$) and a \emph{unique} contractive (weak) morphism $T=(T_t):E^2 \to E^1$ such that $ \psi_t(a)=\la \xi_t^1,T_ta\xi_t^2\ra$ for all $a\in \B, t\ge 0.$
     \end{theorem}

     \begin{proof}
        We shall prove this extending the same ideas of the proof of Theorem \ref{main-single} to the semigroup level.
       Let $(E=(E_t),\beta=(\beta_{t,s}),\eta^\odot=(\eta_t))$ be the inclusion system associated to $\Phi.$ Note that $E_t$'s are von Neumann $M_2(\B)$-$M_2(\B)$-modules. Let $\mathbb{E}_{ij}:=\1 \otimes E_{ij}\in \B\otimes M_2,$ where $E_{ij}$'s are the matrix units in $M_2.$ Let $ \hat{E}_t^i:=\mathbb{E}_{ii} E_t\se E_t,i=1,2.$ Then $\hat{E}_t^i$'s are SOT closed $M_2(\B)$-submodules of $E_t$ such that $E_t=\hat{E}_t^1\oplus \hat{E}_t^2$ for all $t\ge 0.$
        Let  $\eta_t^i:=\mathbb{E}_{ii}\eta_t\mathbb{E}_{ii}\in \hat{E}_t^i,$ $ i=1,2.$  Then we have (as in the proof of Theorem \ref{main-single})
        \begin{equation}
        \eta_t=\eta_t^1+\eta_t^2\text{ with }\la \eta_t^1,\eta_t^ 2\ra=0\text{ and } \eta_t^i=\mathbb{E}_{ii}\eta_t=\eta_t\mathbb{E}_{ii}\text{\quad for all } t\ge0,i=1,2.\label{eq-pf-5}
        \end{equation}
       As $\beta_{t,s}:E_{t+s}\to E_t\odot E_s$ are the canonical maps $\eta_{t+s}\mapsto \eta_t\odot \eta_s,$   using \eqref{eq-pf-5} we have,
         \begin{equation}\label{eq-consistency}
         \beta_{t,s}(\eta^i_{t+s})=\beta_{t,s}(\mathbb{E}_{ii}\eta_{t+s}\mathbb{E}_{ii})=\mathbb{E}_{ii}\eta_t\odot\eta_s\mathbb{E}_{ii}=\eta_t^i\odot\eta_s^i\quad \text{ for  } t,s\ge 0, i=1,2.
         \end{equation}

        Consider the von Neumann $\B$-$\B$-modules $E_t^{(\B)}$ (as described in Remark \ref{obs-new-module}) and the von Neumann $\B$-modules $\hat{E}_t^ {i(\B)}$ (see Proposition \ref{new module}).  Notice that $\hat{E}_t^ {i(\B)}$ is also a von Neumann $\B$-$\B$-module for $i=1,2$ with the left action of $\B$ given by
         \begin{equation}
        a [x]:=[\begin{pmatrix}
        a &0\\0&a
        \end{pmatrix}x]\quad\text{for } a\in\B, x\in \hat{E}_t^i.
        \end{equation}
       Then, we have $E_t^{(\B)}\simeq \hat{E}_t^{1(\B)}\oplus \hat{E}_t^{2(\B)}$ (as two-sided von Neumann modules) for all $t\ge 0$ (as in the proof of Theorem \ref{main-single}). Let $\xi_t^i=[\eta_t^i]\in \hat{E}_t^{i(\B)}, i=1,2.$ Then  for $a\in \B,i=1,2,$ we have
        $$\la \xi_t^i,a\xi_t^i\ra= \sum_{r,s=1}^2\left\la \mathbb{E}_{ii}\eta_t, \begin{pmatrix} a&0\\0&a\end{pmatrix}\mathbb{E}_{ii}\eta_t\right\ra_{r,s}=\sum_{r,s=1}^2\Phi_t\left( \mathbb{E}_{ii}\begin{pmatrix} a&0\\0&a\end{pmatrix}\mathbb{E}_{ii}\right)_{r,s}=\phi_t^i(a).$$
        Therefore, $(\hat{E}_t^{i(\B)}, \xi_t^i)$ is a  GNS-representation (not necessarily minimal) for $\phi_t^i,i=1,2.$  Let $E_t^i= \overline{\spn}^s \B\xi_t^i\B\subseteq \hat{E}_t^{i(\B)}$ be the minimal GNS-module for $\phi_t^i$ for $ i=1,2.$  Let  $\beta_{t,s}^i:E_{t+s}^i\to E_t^i\odot E_s^i$ be the canonical maps (as in Remark \ref{eg-incl-sys}) given by $$\xi_{t+s}^i\mapsto \xi_t^i\odot \xi_s^i\quad\text{for } t,s\ge 0, i=1,2,$$ so that $(E^i=(E_t^i),\beta^i=(\beta_{t,s}^i),\xi^{\odot i}=(\xi^i_t))$ is the inclusion system associated to $\phi^i, i=1,2.$
        (Equation \eqref{eq-consistency} shows that, we get the inclusion systems associated to $\phi^i$'s in a canonical way from the  inclusion system associated to $\Phi$.)

        Let $V_t^i:E_t^i\to \hat{E}_t^{i(\B)}$ be the inclusion maps and let $U_t:\hat{E}_t^{2(\B)}\to\hat{E}_t^{1(\B)}$ be defined by $$U_t[w]=[\mathbb{E}_{12}w]\quad\text{for }w\in \hat{E}_t^2.$$ Then $V_t^i$'s are adjointable, bilinear isometries and  $U_t$'s are bilinear unitaries (as in the proof of Theorem \ref{main-single}).   Take  $T_t:=V_t^{1*}U_tV_t^2.$ Then  $T_t: E_t^2 \to E_t^1$ is an adjointable, bilinear contraction such that for $a\in \B,$
        \begin{align*}
        \la \xi_t^1,T_ta\xi_t^2\ra
        &=\la\xi_t^1,V_t^{1*}U_tV_t^2a\xi_t^2\ra
        =\la V_t^{1}[\eta_t^1],U_tV_t^2a[\eta_t^2]\ra=\left\la [\eta_t^1],U_t[\begin{pmatrix}
        a & 0\\0&a
        \end{pmatrix}\eta_t^2]\right\ra\\
        &=\left\la [\eta_t^1],[\begin{pmatrix}
        0 & a\\0&0
        \end{pmatrix}\eta_t^2]\right\ra=\sum_{i,j=1}^2\left\la\eta_t^1,\begin{pmatrix}
        0&a\\0& 0
        \end{pmatrix} \eta_t^2\right\ra_{i,j}\\
        &=\sum_{i,j=1}^2\left( \mathbb{E}_{11}\Phi_t\begin{pmatrix}
        0&a\\0& 0
        \end{pmatrix}\mathbb {E}_{22}\right)_{i,j}=\sum_{i,j=1}^2\begin{pmatrix}
        0&\psi_t(a)\\0& 0
        \end{pmatrix}_{i,j}
        =\psi_t(a).
        \end{align*}
       For $a,b,c,d\in\B,$
        \begin{align*}
        \la a\xi_{t+s}^1b,T_{t+s}(c\xi_{t+s}^2d)\ra
        &=b^*\psi_{t+s}(a^*c)d\\
        &=b^*\psi_s(\psi_t(a^*c))d\\
        &=b^*\psi_s(\la \xi_t^1,a^*cT_t\xi_t^2\ra) d\\
        &=b^*\la \xi_s^1,T_s(\la \xi_t^1,a^*cT_t\xi_t^2\ra \xi_s^2)\ra  d\\
        &=b^*\la \xi_s^1,\la \xi_t^1,a^*cT_t\xi_t^2\ra T_s\xi_s^2\ra  d\\
        &=b^*\la  \xi_t^1\odot \xi_s^1,a^*c(T_t\odot T_s)(\xi_t^2\odot \xi_s^2)\ra d \\
        &=\la \beta_{t,s}^1(a \xi_{t+s}^1b ),(T_t\odot T_s)\beta_{t,s}^2(c\xi_{t+s}^2d)\ra \\
        &=\la a \xi_{t+s}^1b ,\beta_{t,s}^{1*}(T_t\odot T_s)\beta_{t,s}^2(c\xi_{t+s}^2d)\ra,
        \end{align*}
        shows that $T:=(T_t)_{t\ge 0}$ is a morphism of inclusion systems  from $(E^2,\beta^2)$ to $(E^1,\beta^1).$

        To prove the uniqueness of $T$, let  $T'=(T'_t)_{t\ge 0}$ be another morphism of inclusion systems  from $(E^2,\beta^2,\xi^{\odot 2})$ to $(E^1,\beta^1,\xi^{\odot 1})$ such that $ \psi_t(a)=\la \xi_t^1,T'_t(a\xi_t^2)\ra$ for all $a\in \B, t\ge 0,$ then
        $$  \la a_1\xi_t^1 b_1 , T(a_2\xi_t^2 b_2) \ra
        =b_1^*\psi_t(a_1^*a_2)b_2
        =\la a_1\xi_t^1 b_1 , T'(a_2\xi_t^2 b_2) \ra$$
        for $a_1,a_2,b_1,b_2 \in \B$ and hence $T_t=T'_t$ for all $t\ge 0.$
     \end{proof}

 \begin{example}
    Let $\B$ be a von Neumann algebra. Let $E$ be a von Neumann $M_2(\B)$-$M_2(\B)$-module. Take $\beta=\begin{pmatrix}
    \beta_1 &0\\0&\beta_2
    \end{pmatrix}$ in $M_2(\B)$ and $\zeta\in E$ such that $\zeta=\mathbb{E}_{11}\zeta\mathbb{E}_{11}+\mathbb{E}_{22}\zeta\mathbb{E}_{22},$ where $\mathbb{E}_{ij}=\1\otimes E_{ij}\in\B\otimes M_2$ and $\{E_{ij}\}_{i,j=1}^2$ are the matrix units in $M_2.$

    Let $\xi^\odot(\beta,\zeta)=(\xi_t{(\beta,\zeta)})_{t\in \R_+}\in \IG^\odot (E),$ the product system of time ordered Fock module over $E,$ where the component $\xi_t^n$ of $\xi_t{(\beta,\zeta)}\in \IG_t(E)$ in the $n$-particle $(n>0)$ sector is defined as
     \begin{equation}
      \xi_t^n(t_n,\dots , t_1)=e^{(t-t_n)\beta}\zeta\odot e^{(t_n-t_{n-1})\beta}\zeta\odot \cdots\odot e^{(t_2-t_1)\beta}\zeta e^{t_1\beta}.
     \end{equation}
        and $\xi_t^0=e^{t\beta}.$ Then it follows from \cite[Theorem  3]{LS} that, $\xi^\odot(\beta,\zeta)$ is a unit for the product system $\IG^\odot(E).$ Further if $\Phi_t^{(\beta,\zeta)}:M_2(\B)\to M_2(\B)$ is defined by
        \begin{equation}
        \Phi_t^{(\beta,\zeta)}(A)=\la \xi_t{(\beta, \zeta)}, A\xi_t{(\beta,\zeta)}\ra\quad \text{for } A\in M_2(\B),
        \end{equation}
        then $\Phi:=(\Phi_t)_{t\ge 0}$ is a uniformly continuous CP-semigroup on $M_2(\B),$ with bounded generator
        \begin{equation}
        L(A)=L^{(\beta,\zeta)}(A)=A\beta+\beta^* A+\la\zeta,A\zeta\ra \quad \text{for } A\in M_2(\B).
        \end{equation}

    Let
     $\zeta_i=\mathbb{E}_{ii}\zeta\mathbb{E}_{ii},i=1,2,$ then $\zeta=\zeta_1+ \zeta_2,\la\zeta_1,\zeta_2\ra=0.$ Let $\tau:M_2(\B)\to M_2(\B)$ be defined by $\tau(A)=\la \zeta, A\zeta\ra , A\in M_2(\B),$ then $\tau$ is a block CP map, say $\tau=\begin{pmatrix}
     \tau_{11}&\tau_{12}\\\tau_{12}^*&\tau_{22}
     \end{pmatrix}.$

     Note that $(E^{(\B)},[\zeta_i])$ is a GNS-representation for $\tau_{ii}, i=1,2,$ where $E^{(\B)}$ is the von Neumann $\B$-$\B$-module as described in  Remark \ref{obs-new-module}.

     Let $E_i=\overline{\spn}^s\B [\zeta_i]\B\se E^{(\B)}$  be the minimal GNS-representation for $\tau_{ii},i=1,2$ and let  $T:E_2\to E_1$ be the unique bilinear, adjointable contraction such that $\tau_{12}(a)=\la [\zeta_1],Ta[\zeta_2]\ra$ as given in Theorem \ref{main-single}.  Therefore, we have
    $$\tau(A)=\la \zeta,A\zeta\ra=\begin{pmatrix}
    \la [\zeta_1], a_{11}[\zeta_1]\ra&\la [\zeta_1], Ta_{12}[\zeta_2]\ra\\\la [\zeta_2], T^*a_{21}[\zeta_1]\ra& \la [\zeta_2], a_{22}[\zeta_2]\ra
    \end{pmatrix}, \text{ for }A=\begin{pmatrix}
    a_{11}&a_{12}\\a_{21}&a_{22}
    \end{pmatrix}\in M_2(\B)$$
    and hence
    \begin{align*}
    L(A)&=A\beta+\beta^* A+\la\zeta,A\zeta\ra \\
    &=\begin{pmatrix}
    a_{11}\beta_1+\beta_1^*a_{11}+\la [\zeta_1], a_{11}[\zeta_1]\ra
    &a_{12}\beta_2+\beta_1^*a_{12}+\la [\zeta_1], Ta_{12}[\zeta_2]\ra\\
    a_{21}\beta_1+\beta_2^*a_{21}+\la [\zeta^ 2], T^*a_{21}[\zeta_1]\ra
    &a_{22}\beta_2+\beta_2^*a_{22}+ \la [\zeta_2], a_{22}[\zeta_2]\ra
    \end{pmatrix}\\
    &=\begin{pmatrix}
    L_{11}^{(\beta_1,[\zeta_1])}(a_{11})
    &L_{12}^{(\beta_1, \beta_2,[\zeta_1],[\zeta_2], T)}(a_{12})\\
    L_{21}^{(\beta_1, \beta_2,[\zeta_1],[\zeta_2],T)}(a_{21})
    &L_{22}^{(\beta_2,[\zeta_2])}(a_{22})
    \end{pmatrix},
    \end{align*}
     where
    \begin{equation*}
    L_{ii}(a)=L_{ii}^{(\beta_i,[\zeta_i])}(a)=a\beta_i+\beta_i^*a+\la[\zeta_i],a[\zeta_i]\ra ,\quad i=1,2,
    \end{equation*} and
    \begin{equation}
    L_{12}(a)=L_{12}^{(\beta_1, \beta_2,[\zeta_1],[\zeta_2],T)}(a)=a\beta_2+\beta_1^*a+\la [\zeta_1], Ta[\zeta_2]\ra_\B,\label{L_{12}(1)}
    \end{equation}
    $$L_{21}(a)= L_{12}(a^*)^*,$$
    for $a\in\B.$ Therefore, for $A=\begin{pmatrix}
    a_{11}&a_{12}\\a_{21}&a_{22}
    \end{pmatrix}\in M_2(\B),$
    \begin{equation*}
    \Phi_t(A)=e^{tL(A)}=\begin{pmatrix}
    e^{tL_{11}^{(\beta_1,[\zeta_1])}(a_{11})}
    &e^{tL_{12}^{(\beta_1, \beta_2,[\zeta_1],[\zeta_2], T)}(a_{12})}\\
    e^{tL_{21}^{(\beta_1, \beta_2,[\zeta_1],[\zeta_2],T)}(a_{21})}
    &e^{tL_{22}^{(\beta_2,[\zeta_2])}(a_{22})}
    \end{pmatrix}.
    \end{equation*}

    Now note that the inclusion system $(E^i=(E_t^i),\xi^i=(\xi_t^i(\beta_i,[\zeta_i])))$ associated to $\phi^i=(e^{tL_{ii}^{(\beta_i,[\zeta_i])}})_{t\ge 0}$ is a subsystem of the product system of time-ordered Fock module $\IG^\odot (E_i)$ over $E_i, i=1,2.$

   Let $w=(w_t)_{t\ge 0}$ be the  contractive morphism  from $(E^2,\xi^2)$ to $(E^1,\xi^1)$ such that
    \begin{equation}
    e^{tL_{12}}(a)=\la \xi_t^1(\beta_1,[\zeta_1]), aw_t(\xi_t^2(\beta_2,[\zeta_2]))\ra,\quad \text{for all } a\in \B.  \label{eq-sgp-cor}
    \end{equation}
    As any morphism maps a unit to a unit we have
    \begin{equation}
    w_t(\xi_t^2(\beta_2,[\zeta_2]))=\xi_t^1(\gamma_w(\beta_2,[\zeta_2]),\eta_w(\beta_2,[\zeta_2]))\label{eq-unit-unit}
    \end{equation}for some $\gamma_w,\eta_w:\B\times E_2\to \B\times E_1.$
    Hence from \eqref{eq-sgp-cor} and \eqref{eq-unit-unit} we have
    \begin{equation*}
    e^{tL_{12}}(a)=\la \xi_t^1(\beta_1,[\zeta_1]), \xi_t^1(\gamma_w(\beta_2,[\zeta_2]),\eta_w(\beta_2,[\zeta_2]))\ra.
    \end{equation*}
    Now by differentiating \eqref{eq-sgp-cor}, we get
    \begin{equation}
    L_{12}(a)=\la \zeta_1,a \eta_w(\beta_2,[\zeta_2])\ra+a\gamma_w(\beta_2,[\zeta_2])+\beta_1^*a\label{L_{12}(2)}.
    \end{equation}
    Therefore as \eqref{L_{12}(1)}=\eqref{L_{12}(2)} we have
    $\gamma_w(\beta_2,[\zeta_2])=\beta_2$ and $ \eta_w(\beta_2,[\zeta_2])=T[\zeta_2].$
    Thus, the unique morphism $(w_t)$ is given by $$w_t\xi_t^2(\beta_2,[\zeta_2])=\xi_t^1(\beta_2, T[\zeta_2]).$$
 \end{example}

 \subsection{$E_0$-dilation of  block quantum Markov semigroups}\label{subsec-2}

     In this subsection we shall prove that if we have a block QMS on a unital $C^*$-algebra then the $E_0$-dilation constructed in \cite{BS} is also a semigroup of block maps.

     Let $\B$ be a unital $C^*$-algebra. Let $p\in \B$ be a projection. Denote $p'=\1-p.$ Let $\Phi=(\Phi_t)_{t\ge 0}$ be a block QMS on $\B$ with respect to $p.$

     (We have some changes in the notations from \cite{BS}: $\E_t\leadsto E_t, 
      E_t\leadsto \E_t, E\leadsto \E$)

     Let $(E=(E_t),\xi^\odot=(\xi_t))$ be the inclusion system associated to $\Phi.$ Recall  from \cite[Sections 4, 5]{BS} that
     \[
     \begin{tikzcd}
     (E_t,\xi_t)\arrow[d, "\text{first inductive limit}"] \\
     (\E_t,\xi^t)\arrow["\text{second inductive limit}", d]\\
     (\E,\xi)
     \end{tikzcd}\]
     That is, we have a $\B$-module $\E$ with $\E\simeq \E\odot \E_t,$ a representation $j_0:\B\to \Ba(\E)$ ($b\mapsto \ket{\xi}b\bra{\xi}$) and
     endomorphisms  $\vartheta_t:\B^a(\E)\to \B^a(\E)$ defined by $\vartheta_t(a)=a\odot \id_{\E_t}$ such that  $(\vartheta_t)_{t\ge 0}$ is an $E_0$-dilation of $(\Phi_t)_{t\ge 0}.$ Moreover, we have the Markov property
     \begin{equation}
     j_0(\1) \vartheta_t(j_0(x))j_0(\1)=j_0(\Phi_t(x)),\quad  x\in \B.\label{eq-Markov-property}
     \end{equation}
     This  implies that $j_0(\1) \vartheta_t(j_0(\1))j_0(\1)=j_0(\Phi_t(\1))=j_0(\1).$
     Since $j_0(\1)$ is a projection, we have $j_0(\1)\le \vartheta_t(j_0(\1))$ and hence  $(\vartheta_t(j_0(\1)))_{t\ge 0}$ is an increasing family of projections. Hence it converges in SOT.
     Now if $k_s:\E_s\to \E$ are the canonical maps ($x_s\mapsto\xi\odot x_s$)  then
     \begin{equation}
        \overline{\spn}~k_s(\E_s)=\E.\label{eq-span-1}
     \end{equation}
      Hence  $\vartheta_t(j_0(\1))(\E_t)=\vartheta_t(j_0(\1))(\xi \odot\E_t)=(\ketbra{\xi}{\xi}\odot \id_{\E_t})(\xi
       \odot\E_t)=\xi \odot\E_t$ shows that $\vartheta_t(j_0(\1))_{t\ge 0}$ is converging in SOT to $\id_\E,$ the identity on $\E.$

     Now  for $q=p$ or $p',$ consider $\vartheta_t(j_0(q))=\vartheta_t(|\xi\ra q \la \xi|).$ Note that since $\Phi$ is a unital block semigroup $\Phi_t(q)=q$ for $q=p,p'.$ Hence by the Markov property \eqref{eq-Markov-property} we have
     \begin{equation}
     j_0(\1) \vartheta_t(j_0(q))j_0(\1)=j_0(\Phi_t(q))=j_0(q), \quad \text{for } q=p,p'. \label{mark-rel}
     \end{equation}

     Note that $j_0(\1)=j_0(p)+j_0(p')$ and $j_0(p)j_0(p')=j_0(p')j_0(p)=0.$ Hence multiplying by $j_0(q)$ on both sides of Equation \eqref{mark-rel} we get
     \begin{equation*}
     j_0(q) \vartheta_t(j_0(q))j_0(q)=j_0(q),\quad \text{for } q=p,p'.
     \end{equation*}
         Since $j_0(q)$ is a projection, we have $j_0(q)\le \vartheta_t(j_0(q))$  for all $t,$ hence $\vartheta_s(j_0(q))\le \vartheta_t(j_0(q))$ for $s\le t.$ Therefore $(\vartheta_t(j_0(q)))_{t\ge 0}$ is an increasing family of projections in $\B^a(\E). $ Say  $(\vartheta_t(j_0(p)))_{t\ge 0}$ converges to $P.$ Then as $(\vartheta_t(j_0(\1))_{t\ge 0}$ converges to $\id_\E,$ $(\vartheta_t(j_0(p')))_{t\ge 0}$ will converge to $P'=\id_\E-P.$ Note that  we have   $PP'=0$ and
     \begin{equation}\label{eq-span-2}
        \overline{\spn}^s \vartheta_t(j_0(p))(\E)=P(\E)\quad \text{and}\quad\overline{\spn}^s \vartheta_t(j_0(p'))(\E)=P'(\E).
     \end{equation}
    Thus, we have $\E=\E^{(1)}\oplus\E^{(2)}$ where $\E^{(1)}=P(\E)$  and  $\E^{(2)}=P'(\E).$
\begin{lemma} \label{lem-tec-sec-4.2}
        $P(\E_t)=\vartheta_t(j_0(p))(\E_t)$ and $P'(\E_t)=\vartheta_t(j_0(p'))(\E_t)$ for all $t\ge 0.$
\end{lemma}

 \begin{proof}
Fix $t\ge0.$ It is enough to prove for $q=p,p'$ that   $\vartheta_s(j_0(q))(x)=\vartheta_t(j_0(q))(x)$ if $x\in \E_t,s\ge t.$ Note that (since $\la \xi^t, q\xi^t\ra= \Phi_t(q)=q$ for $q=p$ or $p'$) we have
    \begin{equation*}
        \norm{p\xi^t-p\xi^tp}^ 2=\norm{p\xi^tp'}^2
        =\norm{p'\Phi_t(p)p'}
        =0=\norm{p\Phi_t(p')p}=\norm{p'\xi^tp}^2=\norm{\xi^tp-p\xi^tp}^ 2.\label{eq-lem-tec-pf-1}
    \end{equation*}
   This  implies that  $p\xi^t=\xi^tp=p\xi^tp.$ Similarly we have  $p'\xi^t=\xi^tp'=p'\xi^tp'.$

     Let $q=p$ or $p'$ and let $s\ge t.$ If $x\in\E_t,$ then $\xi^{s-t}\odot x\in \E_s$ and
     \begin{align*}
     \vartheta_s(j_0(q))(x)&=(\ket{\xi}q\bra{\xi}\odot \id_{\E_s})(\xi\odot \xi^{s-t}\odot x)=\xi\odot q\xi^{s-t}\odot  x=\xi\odot\xi^{s-t}q \odot x\\&=\xi\odot\xi^{s-t} \odot qx=\xi\odot qx=\vartheta_t(j_0(q))(x).
     \end{align*}
 \end{proof}

     We have from \cite[Theorem 5.4]{BS} that
     \begin{equation}
     \E\simeq\E\odot\E_t, \text{\quad for all } t\ge 0.
     \end{equation}

     Now we shall prove a similar result for $\E^{(i)}$'s by recalling the proof of this
     result. It is important to note that we are not getting something like
     $\E^{(i)} = \E^{(i)}\odot \E _t^{(i)}$, and we have not even
     bothered to define $ \E _t^{(i)}$.

     \begin{lemma}
        $\E^{(i)}\simeq \E^{(i)}\odot \E_t,$ for $ i=1,2, t\ge 0.$
     \end{lemma}

     \begin{proof}
       Let $k_t:\E_t\to \E$ be the canonical maps (isometries). Then  $u_t:\E\odot\E_t\to \E$ defined by
        \begin{equation}
        u_t(k_s(x_s)\odot y_t)=k_{s+t}(x_s\odot y_t)
        \end{equation}
        for $x_s\in \E_s, y_t\in \E_t,$ is a unitary (\cite[Theorem 5.4]{BS}). Hence,  we have $\E\simeq\E\odot\E_t.$ Since $\E=\E^{(1)}\oplus\E^{(2)},$ we have, $\E\simeq \E^{(1)}\odot\E_t\oplus\E^{(2)}\odot\E_t.$

         We shall prove that, the restriction of this unitary $u_t$ to $\E^{(i)}\odot\E_t$ is a unitary from $\E^{(i)}\odot\E_t$ onto $\E^{(i)}.$  It is enough to prove that $u_t(\E^{(i)}\odot\E_t)\subseteq \E^{(i)}.$ To prove this, (from \eqref{eq-span-1}, \eqref{eq-span-2} and Lemma \ref{lem-tec-sec-4.2})  it is sufficient to prove that $u_t(\vartheta_s(j_0(p))k_s(\E_s)\odot \E_t)\subseteq \E^{(1)}$ and $u_t(\vartheta_s(j_0(p'))k_s(\E_s)\odot \E_t)\subseteq \E^{(2)}.$ To prove this consider for  $q=p$ or $p' $ and $x_s\in \E_s$
        \begin{align*}
        u_t(\vartheta_s(j_0(q))k_s(x_s)\odot y_t)&=u_t((\xi q\odot x_s)\odot y_t)=u_t((\xi \odot q x_s)\odot y_t)\\
        &= u_t(k_s(q x_s)\odot y_t)\\
        &=k_{s+t}(q x_s\odot y_t)\\
        &=\xi\odot q x_s\odot y_t=\xi q\odot x_s\odot y_t\\
        &=\vartheta_{s+t}(j_0(q))k_{s+t}(x_s\odot y_t),
        \end{align*}
       which is in $\E ^{(1)}$ if $q=p$ and is in $\E ^{(2)}$ if
       $q= p'.$
     \end{proof}

     \begin{theorem}\label{thm-block-E-zero}
        The $E_0$-dilation $\vartheta=(\vartheta_t)_{t\ge 0}$ of $\Phi$ is a semigroup of block maps with respect to the projection $P$ defined above.
     \end{theorem}

 \begin{proof}
    As $\E=\E^{(1)}\oplus\E^{(2)},$ we have
    \[\Ba(\E)=\begin{pmatrix}
    \Ba(\E^{(1)})&\Ba(\E^{(2)},\E^{(1)})\\ \Ba(\E^{(1)},\E^{(2)})&\Ba(\E^{(2)})
    \end{pmatrix}.
    \]

    For any $i,j\in\{1,2\},$ let $a\in\Ba(\E^{(i)},\E^{(j)}),$ then  $$\vartheta_t(a)=a\odot \id_{\E_t}\in \Ba(\E^{(i)}\odot \E_t,\E^{(j)}\odot \E_t)=\Ba(\E^{(i)},\E^{(j)}). $$
    Therefore $\vartheta_t$ acts block-wise.
 \end{proof}


     \section{Lifting of morphisms}\label{lift-morphsm}

     In this section we will show that any (weak) morphism between two inclusion systems of  von Neumann $\B$-$\B$-modules can be always lifted as a morphism between the product systems generated by them.

     We shall introduce some notations and results from \cite{BS} and \cite{BM}.     For all $t> 0$ we define
     \begin{equation}
     \J_t:=\{\ft=(t_n,\dots,t_1)\in \T^n:t_i>0 ,|\ft|=t, n\in \mathbb{N}\}
     \end{equation}
     and for $\fs=(s_m,\dots,s_1)\in\J_s$ and $\ft=(t_n,\dots,t_1)\in\J_t$ we define the \emph{joint tuple} $\fs\smile \ft\in \J_{s+t}$ by
     $$\fs\smile\ft=((s_m,\dots,s_1),(t_n,\dots,t_1))=(s_m,\dots,s_1,t_n,\dots,t_1).$$

     We have a partial order $``\ge"$ on $\J_t$ as follows:  $\ft\ge \fs=(s_m,\dots, s_1),$ if for each $j$ $(1\le j\le m)$ there are (unique) $\fs_j\in \J_{s_j}$ such that $\ft=\fs_m\smile \dots \smile \fs_1$ (In  this case we also write $\fs\le \ft,$ to mean $\ft\ge \fs$).

     For $t=0$ we extend the definition of $\J_t$ as $\J_0=\{()\},$ where $()$ is the empty tuple. Also for $\ft\in \J_t$ we put $\ft\smile ()=\ft=()\smile \ft.$

     Now we will describe the construction of product system generated by an inclusion system of von Neumann $\B$-$\B$-modules using the inductive limits. (This construction holds also for Hilbert $\B$-$\B$-modules  along the same lines, but as we are going to prove the lifting theorem only for von Neumann $\B$-$\B$-modules,
      we confine ourselves to von Neumann modules).

     Let $(E=(E_t)_{t\in \T},\beta=(\beta_{s,t})_{s,t\in \T})$ be an inclusion system of von Neumann $\B$-$\B$-modules. Fix $t\in \T.$ Let $E_\ft:=E_{t_n}\odot \cdots \odot E_{t_1}$ for $\ft=(t_n,\dots, t_1)\in\J_t.$ For all $\ft=(t_n,\dots, t_1)\in\J_t$  we define $\beta_{\ft (t)}:E_t\to E_\ft$ by
     $$\beta_{\ft (t)}=(\beta_{t_n,t_{n-1}}\odot \id)(\beta_{t_n+t_{n-1},t_{n-2}}\odot \id)\dots (\beta_{t_n+\dots +t_3,t_2}\odot \id)\beta_{t_n+\dots +t_2, t_1},$$
     and  for $\ft=(t_n,\dots, t_1)=\fs_m\smile\dots \smile\fs_1\ge\fs=(s_m,\dots,s_1)$ with $|\fs_j|=s_j,$ we  define $\beta_{\ft \fs}:E_\fs\to E_\ft$ by $$\beta_{\ft \fs}=\beta_{\fs_m (s_m)}\odot\dots \odot\beta_{\fs_1 (s_1)}.$$ Then it is clear from the definitions that $ \beta_{\ft\fs},\ft\ge\fs$ are bilinear isometries and $\beta_{\ft \fs}\beta_{\fs\fr}=\beta_{\ft\fr}$ for $\ft\ge\fs\ge \fr. $ That is, the family $(E_\ft)_{\ft\in\J_t}$ with $(\beta_{\ft\fs})_{\fs\le \ft}$  is an inductive system of von Neumann $\B$-$\B$-modules. Hence the inductive limit $\E_t=\underset{\ft\in\J_t}{\ilim}~E_\ft$ is also a von Neumann $\B$-$\B$-module and the canonical mappings $i_\ft: E_\ft\to \E_t$ are  bilinear  isometries (cf. \cite[Proposition 4.3]{BS}).

     For $\fs\in \J_s,\ft\in \J_t$ it is clear that $E_\fs\odot E_\ft=E_{\fs\smile\ft}.$  Using this observation  we define  $B_{st}: \E_s\odot \E_t\to \E_{s+t}$ by
     $$B_{st}(i_\fs x_\fs \odot i_\ft y_\ft)=i_{\fs\smile \ft}(x_\fs\odot y_\ft)\text{ for } x_\fs\in E_\fs, y_\ft\in E_\ft,\fs\in \J_s, \ft \in \J_t.$$
     Then  $(\E=(\E_t)_{t\in \T},B=(B_{st})_{s,t\in\T})$ forms a product system (cf. Bhat and Skeide \cite[Theorem 4.8 and page 41]{BS}).

     \begin{definition}
        Given an inclusion system $(E,\beta),$ the product system $(\E,B)$ described above is called the product system generated by the inclusion system $(E,\beta).$
     \end{definition}

     We recall the following: Let $\B$ be a von Neumann algebra on a Hilbert space $\G.$ Let $E$ be a von Neumann $\B$-module. Then $\H=E\odot \G$ is a Hilbert space such that $E\se \SB(\G,\H)$  via  $E\ni x\mapsto L_x\in \SB(\G,\H),$ where $L_x:\G\to \H$ is defined by $L_x(g)=x\odot g$ for $g\in\G.$ Note that $E$ is strongly closed in $\SB(\G,\H).$ Sometimes we write  $xg$ instead of $x\odot g$ with the above identification in mind.

      \begin{remark}\label{rem-tech}
        Let $(\E, B)$ be the product system generated by the inclusion system $(E, \beta )$ on a von Neumann algebra
         $\B\subseteq\BG.$ Let $i_\ft:E_\ft\to \E_t, \ft\in \mathbb{J}_t$ be the canonical bilinear isometries.  Then $i_\ft i_\ft^*$
         increases to identity in strong operator topology, that is,
         for all  $x\in \E_t$ and $ g\in \G,$ we have
          \begin{equation}
             \lim\limits_{\ft\in \mathbb{J}_t}\norm{xg-i_\ft i_\ft^*(x)g}=0.\label{eq-rem-tech}
          \end{equation}
     \end{remark}

 Now we shall prove the lifting theorem almost  along  the same lines of the proof of \cite[Theorem 11]{BM}
 \begin{theorem}\label{thm-lifting}
        Let $\B$ be a von Neumann algebra on a Hilbert space $\G.$
        Let $(E, \beta)$ and $(F, \gamma )$ be two inclusion systems of von Neumann $\B$-$\B$-modules generating two product systems $(\E, B), (\mathcal{F}, C)$ respectively. Let $i,j$ be their respective inclusion maps. Suppose $T:(E, \beta)\to (F, \gamma )$ is a (weak) morphism then there exists a unique morphism $\hat{T}: (\E, B)\to (\mathcal{F}, C)$ such that $T_s=j_s^*\hat{T}_si_s$ for all $s\in \T.$
     \end{theorem}
     \begin{proof}
        Given that $T:(E, \beta)\to (F, \gamma )$ is a morphism. Let $k$ be such that $\norm{T_s}\le e^{ks}$ for all $s\in \T.$ For $\fs=(s_n,...,s_1)\in \mathbb{J}_s,$ define $T_\fs :E_\fs\to F_\fs$ by $T_\fs=T_{s_n}\odot \cdots \odot T_{s_1}.$ Let $i_\fs: E_\fs\to \E_s$ and  $j_\fs: F_\fs\to \mathcal{F}_s$ be the canonical bilinear isometries. Then for $\fs \le \ft$ in $\J_s$ we have
        \begin{equation}
        \gamma_{\ft\fs}^*T_\ft\beta_{ \ft\fs}=T_\fs.\label{eq-pf-thm-1}
        \end{equation}

        Consider for $\fs\in \mathbb{J}_s, \Phi_ \fs:=j_\fs T_\fs i_\fs^*.$ Set $P_\fs=j_\fs j_\fs^*$ and $Q_\fs=i_\fs i_\fs^*.$ Then by Remark \ref{rem-tech} $(P_\fs)_{\fs\in \J_s}$ and $(Q_\fs)_{\fs\in \J_s}$ are families of increasing projections. Now for $\fr\le \fs,$ $i_\fr=i_\fs\beta_{\fs\fr},$  $j_\fr=j_\fs\gamma_{\fs\fr} $ implies that $\beta_{\fs\fr}=i_\fs^*i_\fr,\gamma_{\fs\fr}=j_\fs^*j_\fr,$ hence it follows from \eqref{eq-pf-thm-1} that $P_\fr\Phi_\fs Q_\fr=\Phi_\fr.$

        For all $s\in \T,$ $\E_s\se \SB(\G,\E_s\odot \G)$ and $\F_s\se \SB(\G,\F_s\odot \G).$ Fix $s\in \T.$ Let $x\in \E_s, g\in \G$ and let $\epsilon >0.$ Using \eqref{eq-rem-tech} choose $\fr_0\in \mathbb{J}_s$  such that \begin{equation}
        e^{ks}\norm{Q_{\fr_0}(x)g-xg}<\frac{\epsilon}{3}.\label{eq-pf-0}
        \end{equation}
        Then, for any $\fs\in \mathbb{J}_s,$  we have
        \begin{align}
        \norm{\Phi_\fs(x)g-\Phi_\fs Q_{\fr_0}(x)g}& =\norm{\Phi_\fs(x)\odot g-\Phi_\fs Q_{\fr_0}(x)\odot g}
        \nonumber\\
        &=\norm{(\Phi_\fs\odot \id _\G)(x\odot g-Q_{\fr_0}(x)\odot g)}\nonumber\\
        &\le\norm{\Phi_\fs\odot \id _\G} \norm{xg-Q_{\fr_0}(x) g}\nonumber\\
        &\le e^{ks}~ \norm{xg-Q_{\fr_0}(x) g}<\frac{\epsilon}{3}. \quad\text{ (by \eqref{eq-pf-0})}\label{eq-pf-1}
        \end{align}

        Let $\ft\ge\fs\ge \fr_0\in \mathbb{J}_s.$ As $(P_\fs)_{\fs\in \J_s}$ and $(Q_\fs)_{\fs\in \J_s}$ are increasing families of projections, we have
        \begin{align}
        \norm{\Phi_\ft Q_{\fr_0}(x)g}^2 \nonumber&=\norm{P_\ft \Phi_\ft Q_{\fr_0}(x) g}^2 \nonumber\\
        &= \norm{P_\fs\Phi_\ft Q_{\fr_0}(x) g+(P_\ft-P_\fs) \Phi_\ft Q_{\fr_0}(x) g}^2 \nonumber\\
        &= \norm{P_\fs\Phi_\ft Q_{\fr_0}(x)g}^2+\norm{(P_\ft-P_\fs)\Phi_\ft Q_{\fr_0}(x)g}^2 \nonumber\\
        &= \norm{P_\fs\Phi_\ft Q_\fs Q_{\fr_0}(x)g}^2+\norm{\Phi_\ft Q_{\fr_0}(x)g-P_\fs\Phi_\ft Q_\fs Q_{\fr_0}(x)g}^2 \nonumber\\
        &= \norm{\Phi_\fs Q_{\fr_0}(x)g}^2+\norm{\Phi_\ft Q_{\fr_0}(x)g-\Phi_\fs Q_{\fr_0}(x)g}^2.\label{eq-pf-2}
        \end{align}
        Hence for  $\ft\ge\fs\ge \fr_0\in \mathbb{J}_s, $ we have $\norm{\Phi_\ft Q_{\fr_0}(x)g}^2 \ge \norm{\Phi_\fs Q_{\fr_0}(x)g}^2.$ Also
        $$\norm{\Phi_\fs Q_{\fr_0}(x)g}^2\le \norm{\Phi_\fs Q_{\fr_0}}^2\norm{x}^2 \norm{g}^2\le e^{2ks}\norm{x}^2 \norm{g}^2$$
        for all $\fs\in \mathbb{J}_s.$  Thus  $(\norm{\Phi_\fs Q_{\fr_0}(x)g}^2 )_{\fs\in \mathbb{J}_s}$
        is a Cauchy net, hence choose $\fr_1\in \mathbb{J}_s, \fr_1\ge \fr_0$  such that
        \begin{equation}
        |\norm{\Phi_\ft Q_{\fr_0}(x)g}^2 - \norm{\Phi_\fs Q_{\fr_0}(x)g}^2|<\left(\frac{\epsilon}{3}\right)^2\quad\text{for } \ft\ge\fs\ge \fr_1\ge \fr_0\in \mathbb{J}_s .\label{eq-pf-3}
        \end{equation}

        Therefore for $ \ft\ge\fs\ge \fr_1$ in $\J_s,$  from \eqref{eq-pf-2} and \eqref{eq-pf-3} we have
        \begin{equation}
        \norm{\Phi_\ft Q_{\fr_0}(x)g-\Phi_\fs Q_{\fr_0}(x)g}=|\norm{\Phi_\ft Q_{\fr_0}(x)g}^2 - \norm{\Phi_\fs Q_{\fr_0}(x)g}^2|^\frac{1}{2}<\frac{\epsilon}{3}.\label{eq-pf-4}
        \end{equation}

        Now  for $ \ft\ge\fs\ge \fr_1$ in $\J_s ,$ from \eqref{eq-pf-1} and \eqref{eq-pf-4}  we have
        \begin{align*}
        &\norm{(\Phi_\ft-\Phi_\fs)(x)g}\\&\le\norm{\Phi_\ft(x)g-\Phi_\ft Q_{\fr_0}(x)g}+\norm{\Phi_\ft Q_{\fr_0}(x)g-\Phi_\fs Q_{\fr_0}(x)g}+\norm{ \Phi_\fs Q_{\fr_0}(x)g-\Phi_\fs(x)g}< \epsilon.
        \end{align*}
        Thus  $\lim\limits_{\fs\in \J_s}\Phi_\fs(x)g$ exists. Define $\hat{T}_s(x)g:=\lim\limits_{\fs \in \J_s}\Phi_\fs(x)g$
         for $s>0.$ This defines a bounded bilinear map $\hat{T}_s:\E_s\to \F_s$ for all
         $s\in \T$.

        Now for $\fs\in\J_s$ and for all $ x_\fs\in E_\fs, g\in \G ,$ we have
        $$
        j_\fs^*\hat{T}_si_\fs(x_\fs)g=\lim\limits_{\fr\in\J_s}j_\fs^*\Phi_\fr i_\fs(x_\fs)g=\lim\limits_{\fr\in\J_s}j_\fs^*j_\fr T_\fr i_\fr^*i_\fs(x_\fs)g=\lim\limits_{\fr\in\J_s}\gamma_{\fr\fs}^*T_\fr\beta_{\fr\fs}(x_\fs)g=T_\fs(x_\fs)g.
        $$
        Thus $T_\fs=j_\fs^*\hat{T}_si_\fs$ for all $\fs\in \mathbb{J}_s$ and  $s\in\T.$ In particular $T_s=j_s^*\hat{T}_si_s$ for all $s\in\T.$

        Now we shall prove that $(\hat{T}_t)_{t\in\T}$ is a morphism of product systems.
        For  $\ft\in\mathbb{J}_t,\fs\in \mathbb{J}_s$ and  $x_{\ft}\in E_\ft, x_\fs\in E_\fs , y_{\ft}\in F_\ft, y_\fs\in F_\fs$ consider,
        \begin{align*}
        \la C_{s,t}^*(\hat{T}_s\odot \hat{T}_t)B_{s,t} i_{\fs\smile \ft}(x_\fs\odot x_\ft), j_{\fs\smile \ft}(y_\fs\odot y_\ft)\ra
        &=\la   (\hat{T}_s\odot \hat{T}_t) (i_\fs\odot i_\ft)(x_\fs\odot x_\ft), j_\fs\odot j_\ft(y_\fs\odot y_\ft)\ra\\
        &=\la   \hat{T}_si_\fs x_\fs\odot \hat{T}_ti_\ft x_\ft, j_\fs y_\fs\odot y_\ft)\ra\\
        &=\la   j_\ft^*\hat{T}_ti_\ft x_\ft, \la j_\fs^*\hat{T}_si_\fs x_\fs, y_\fs\ra y_\ft\ra\\
        &=\la    T_\fs x_\fs\odot T_\ft x_\ft, y_\fs\odot y_\ft\ra\\
        &=\la    (T_\fs\odot T_\ft)(x_\fs\odot x_\ft), y_\fs\odot y_\ft\ra\\
        &=\la    T_{\fs\smile \ft}(x_\fs\odot x_\ft), y_\fs\odot y_\ft\ra\\
        &=\la    \hat{T}_{s+t}i_{\fs\smile \ft}(x_\fs\odot x_\ft), j_{\fs\smile \ft}(y_\fs\odot y_\ft)\ra
        \end{align*}
        Thus $ \hat{T}_{s+t}=C_{s,t}^*(\hat{T}_s\odot \hat{T}_t)B_{s,t}$ for all $s,t\in \T.$
     \end{proof}

    \noindent\textbf{Acknowledgments.} The first author thanks J C
    Bose Fellowship and the second author thanks
    NBHM, and the Indian Statistical Institute for research funding.

\end{document}